\documentclass[11pt]{amsart}

\usepackage{amssymb}
\usepackage{bm}
\usepackage[centertags]{amsmath}
\usepackage{amsfonts}
\usepackage{amsthm}
\usepackage{mathrsfs}
\usepackage[usenames]{xcolor}
\usepackage[normalem]{ulem}
\usepackage{comment}
\usepackage{enumitem}
\usepackage{graphicx}
\usepackage[makeroom]{cancel}
\usepackage[all]{xy}

\newcommand{\vcxymatrix}[1]{\vcenter{\xymatrix{#1}}}

\theoremstyle{definition}
\newtheorem{thm}{Theorem}[section]
\newtheorem{dfn}[thm]{Definition}
\newtheorem{lem}[thm]{Lemma}
\newtheorem{prop}[thm]{Proposition}

\newtheorem{rem}[thm]{Remark}
\newtheorem{exa}[thm]{Example}

\newtheorem*{defn*}{Definition}

\newtheorem*{thm*}{Theorem}

\newtheorem*{cor*}{Corollary}

\newtheorem*{prp*}{Proposition}

\newtheorem*{ntt}{Notation}

\newcommand{\N}{\mathbb{N}}
\newcommand{\R}{\mathbb{R}}
\newcommand{\e}{\varepsilon}
\newcommand{\de}{\delta}
\newcommand{\De}{\Delta}
\newcommand{\al}{\alpha}

\newcommand{\la}{\lambda}
\newcommand{\La}{\Lambda}
\newcommand{\Ga}{\Gamma}
\newcommand{\ze}{\zeta}

\newcommand{\cD}{\mathcal D}
\newcommand{\cH}{\mathcal H}

\newcommand{\trn}[1]{{\left\vert\kern-0.25ex\left\vert\kern-0.25ex\left\vert #1 
    \right\vert\kern-0.25ex\right\vert\kern-0.25ex\right\vert}}

\newcommand{\supp}{\text{\rm supp}}
\newcommand{\ie}{{\it i.e.}}

\long\def\symbolfootnote[#1]#2{\begingroup
\def\thefootnote{\fnsymbol{footnote}}\footnote[#1]{#2}\endgroup}

\allowdisplaybreaks

\begin{document}

\title[The space ${L_1(L_p)}$ is primary]{The space $\mathbf{L_1(L_p)}$ is primary for $\mathbf{1<p<\infty}$}

\author{R. Lechner}

\address{R. Lechner, Institute of Analysis, Johannes Kepler University Linz, Altenberger Strasse 69,
  A-4040 Linz, Austria}

\email{richard.lechner@jku.at}

\author{P. Motakis}

\address{P. Motakis, Department of Mathematics and Statistics, York University, 4700 Keele Street, Toronto, Ontario, M3J 1P3, Canada}

\email{pmotakis@yorku.ca}

\author{P.F.X. M\"uller}

\address{P.F.X. M\"uller, Institute of Analysis, Johannes Kepler University Linz, Altenberger
  Strasse 69, A-4040 Linz, Austria}

\email{paul.mueller@jku.at}

\author{Th.~Schlumprecht}

\address{Th.~Schlumprecht, Department of Mathematics, Texas A\&M University, College Station, TX
  77843-3368, USA, and Faculty of Electrical Engineering, Czech Technical University in Prague,
  Zikova 4, 16627, Prague, Czech Republic}

\email{schlump@math.tamu.edu}

\thanks{The first  author was supported by the Austrian Science Foundation (FWF) under
  Grant Number Pr.Nr. P32728. The third author were supported by the Austrian Science Foundation (FWF) under
  Grant Number Pr.Nr. P28352. The fourth author was supported by the National
  Science Foundation under Grant Number  DMS-1764343.}

\date{\today}



\begin{abstract}
  The classical Banach space $L_1(L_p)$ consists of measurable scalar functions $f$ on the unit square for which
  $$\|f\| = \int_0^1\Big(\int_0^1 |f(x,y)|^p dy\Big)^{1/p}dx < \infty. $$
   We show that $L_1(L_p)$  $(1 <  p <  \infty)$ is  primary, meaning that, whenever $L_1(L_p) = E\oplus F$ then either $E$ or $F$ is isomorphic to $L_1(L_p)$.
      More generally we show that   $L_1(X)$ is primary, for a large class of rearrangement invariant Banach function spaces. 
\end{abstract}

\maketitle

\tableofcontents

\section{Introduction}

The decomposition of  normed linear spaces into direct sums
and the analysis of the associated projection operators is central to  important chapters
in the theory of modern and classical Banach spaces.
In a seminal paper J. Lindenstrauss  \cite{JL}  set forth an influential research program, aiming at detailed investigations of complemented subspaces  and  operators on
Banach spaces. 

The  main question adressed by J. Lindenstrauss  was this: 
Which are the spaces $X$ that cannot  be further decomposed into two ``essentially different, infinite dimensional subspaces?
That is to say,   which are the Banach spaces   $X$ that  are not
isomorphic to the direct  sum of
 two  {infinite dimensional} spaces $Y$ and $Z$, where  neither, $Y$ nor $Z$, are isomorphic to $X$?
This condition would be satisfied if $X$ were indecomposable, i.e.,  for any decomposition of $X$ into two spaces, one of them has to be finite dimensional.
Separately, such a space could  be primary, meaning that for any decomposition of $X$ into two spaces, one of them has to be isomorphic to $X.$
The first example of an   indecomposable Banach spaces  was constructed by  T. Gowers and B. Maurey \cite{GM}  who also showed that their space $X_{\rm GM} $ is not primary-- 
 indeed, the infinite dimensional component of $X_{\rm GM} \sim X\oplus Y  $ is not isomorphic to the whole space.   

 While indecomposable spaces play a tremendous role  (\cite{AH,GM,M})  in the present day study of non classical Banach spaces,
a wide variety of  Banach function spaces may  usually be decomposed, for instance
by restriction  to subsets, or by taking conditional expectations etc.
This provides the background for  the program set forth  by J. Lindenstrauss
to determine the ``classical''
 spaces that  are primary.

\subsection {Background and History} The term classical Banach space--while not formally defined--applies certainly to the 
space $C[0,1]$ and to  scalar and vector valued  Lebesgue spaces.
The space of continuous functions was shown to  be primary by J. Lindenstrauss and A. Pelczynski~\cite{LP},  who posed the corresponding problem  for scalar valued $L_p$ spaces.
Its elegant solution,  by  P. Enflo via B. Maurey  \cite{EM}, introduced a groundbreaking method of proof
which applies equally well to each of the   $L_ p $ spaces, $ (1 \le p < \infty)$. 
Later  alternative  proofs  were given by 
D. Alspach P. Enflo E. Odell  \cite{AEO}  for $L_p $ in the reflexive range  $1<p<\infty$ 
  and by P. Enflo and T. Starbird \cite{ES} for $L_1$.
  
Exceptionally deep results on the  decomposition of  Bochner-Lebesgue spaces $L_p(X)$ are due to M. Capon \cite{CAPLX,CAPLL}  who 
obtained  that those spaces are is primary in the following cases.
\begin{enumerate}
\item[-] $X$ is a Banach space with a symmetric basis, and $1 \le p < \infty $. 
  \item[-]   $ X = L_ q  $ where $ 1< q < \infty $ and $ 1< p < \infty $.
  \end{enumerate}
This leaves the  spaces    $ L_1(L_p) $   and $ L_p(L_1) $ among the most prominent examples of
classical Banach spaces for which  primariness is open.

The purpose of the present paper is to prove that $ L_1(L_p) $ is
primary. Our proof works equally well for real and complex valued functions.
Before we turn to describing our work, 
we review in some detail the development of methods  pertaining to the  spaces  $L_p$, and more broadly to rearrangement invariant spaces.

Projections on those spaces are studied effectively  alongside the Haar system  and the
reproducing properties of its block bases. 
The  methods developed for proving that a particular Lebesgue space $L_p$,   is primary may be divided into two basic classes,  depending on whether the Haar system is an
unconditional Schauder basis, or not. 

In case of unconditionality,  the most flexible method goes back to  the work of D. Alspach, P. Enflo, and  E. Odell \cite{AEO}. 
For a  linear operator $T$ on $L_p $ it  yields  a block basis of the Haar system  $\widetilde{h}_I $ and  a bounded sequence of scalars $a_I $
forming an approximate eigensystem of  $T$
such that
\begin{equation}
  \label{eq:13-1-21-1}
   T\widetilde{h}_I = a_I \widetilde{h}_I + \text{a small error}
\end{equation}
and $\widetilde{h}_I$  spans  a complemented copy of the  space $L_p $. 
Thus, when restricted to  $ span \widetilde{h}_I$, the operator  $T$ acts as a bounded Haar multiplier. Since the Haar basis  is unconditional, the Haar multiplier is invertible   if $  |a_I | > \delta $ for some $\delta  > 0 . $

 D. Alspach, P. Enflo, and  E. Odell \cite{AEO} arrive at \eqref{eq:13-1-21-1} by
ensuring that, for  $\varepsilon_{I, J }> 0 $  sufficiently small, 
 the following linearly ordered set of constraints holds true, 
\begin{equation}
  \label{eq:13-1-21-2}
  |\langle  T \widetilde{h}_I, \widetilde{h}_J \rangle | +
  |\langle  \widetilde{h}_I,  T^*\widetilde{h}_J \rangle |
  \le \varepsilon_{I, J } \quad\text{for} \quad I \prec J , 
\end{equation}
where the relation $\prec $ refers to the lexicographic order on the collection of
dyadic intervals. Utilizing that the independent $\{-1,+1\}$-valued Rademacher system $\{r_n\}$
is a weak null sequence in $L_p$, $(1 \le p < \infty )$,   D. Alspach, P. Enflo, and  E. Odell \cite{AEO} obtain, {\em by induction} along $\prec$, the block basis
$ \widetilde{h}_I$ satisfying \eqref{eq:13-1-21-2}.

The Alspach-Enflo-Odell method provides the basic model for the study of operators on function spaces in which the Haar system is unconditional; this applies in particular to rearrangement
invariant spaces  in the work of  W. Johnson, B. Maurey, G. Schechtman, and L. Tzafriri \cite{JMST}, 
and  D. Dosev, W. Johnson, and  G. Schechtman  \cite{DJS}.

In $L_1 $ the Haar system is a Schauder basis but fails to be unconditional. The basic methods for proving that $L_1 $ is primary are due to  P. Enflo via B. Maurey \cite{EM}  on the one hand side and
  P. Enflo and T. Starbird  \cite{ES} on the other hand side.  
For  operators $T$ on $L_1 $ the Enflo-Maurey method yields a block basis of the Haar basis $\widetilde{h}_I $ and  a bounded measurable function $g$,  such that
\begin{equation}
  \label{eq:13-1-21-5}
  (T f) (t) =  g(t)  f(t) + \text{a  small error},
\end{equation}
for $ f \in \mathrm{span}\{  \widetilde{h}_I\}$,  and  $\widetilde{h}_I$  spans  a copy of  $L_1 $.  Thus
the restricted operator $T$ acts as a bounded multiplication operator and is invertible  if $  |g | > \delta $ for some $\delta > 0 . $
The full strength of the proof by  Enflo-Maurey is applied to show that the  representation \eqref{eq:13-1-21-5} holds true.

Enflo-Maurey  \cite{EM}, exhibit in their proof of \eqref{eq:13-1-21-5} a sequence of bounded scalars  
$a_I $ such that 
\begin{equation}
  \label{eq:13-1-21-4}
   T\widetilde{h}_I = a_I \widetilde{h}_I + \text{a very small error}
\end{equation}
Since the  Rademacher system $\{r_n\}$ is a weakly null sequence in $L_1$,  \eqref{eq:13-1-21-4} may  be  obtained directly  by choosing a block basis for which
the constraints \eqref{eq:13-1-21-2}   and
\begin{equation}
  \label{eq:14-1-21-1}
  |\langle  T \widetilde{h}_I, \widetilde{h}_J \rangle | 
  \le 2 \varepsilon_{I, J } \quad\text{for} \quad I \neq J , 
\end{equation}
hold true.
Remarkably,  {\em until very  recently} \cite{LMMS}, eigensystem representations such as \eqref{eq:13-1-21-4} were not exploited  in the context of $L ^1  $, where the
Haar system is not unconditional.

The  powerful precision  of  $L_1$-constructions with dyadic martingales and block basis of the Haar system   is in full  display  in \cite{JMS} and  \cite{T}.
W. Johnson, B.  Maurey and G. Schechtman 
determined  in \cite{JMS} a normalized weakly null sequence in $L_1 $ such that each  of its 
infinite subsequences contains in its span  a block basis of the Haar system    $\widetilde{h}_I$, spanning  a copy of $L_1 . $ Thus 
$L_1 $ fails to satisfy  the unconditional subsequence property,  a problem posed by B. Maurey and H. Rosenthal \cite{MR}.
By contrast M. Talagrand   \cite{T}  constructed a dyadic martingale difference sequence  $g_{n,k} $ such  that neither $X = \overline{\mathrm{span}\, }^{L_1}  \{ g_{n,k} \}$ nor $L_1/X$ contain a copy of $L_1$.

The  investigation  of complemented subspaces in Bochner Lebesgue spaces was initiated  by M.  Capon \cite{CAPLX,CAPLL}  who pushed  hard to further the development of the scalar methods, and
proved that 
  $L_p(X)$ $(1\le p < \infty)$ is primary when  $X$ is a Banach space with a symmetric basis, say $(x_{k}) .$
Specifically, 
   M. Capon \cite{CAPLX} showed, that for an operator $T$ on $L_p(X)$, there exists a block basis of the
  Haar basis $\widetilde{h}_I $, a subsequence of the symmetric basis $(x_{k_n}) $
and a bounded measurable  $g$ such that 
$$ (T(f  \otimes x_{k_n})) (t) =  g(t)
f(t) \otimes x_{k_n}+ \text{a  small error} ,$$
for $ f \in   \mathrm{span} \{\widetilde{h}_I \}$.
Thus on $ \mathrm{span} \{\widetilde{h}_I \} \otimes  \mathrm{span} \{x_{k_n}\}$
the operator $T$ acts like $M_g \otimes Id  $ where $M_g $ is the multiplication operator induced by $g. $
Simultaneously,  M.  Capon shows 
 that
  the tensor products  form an approximate eigensystem,   
$$ T(\widetilde{h}_I\otimes x_{k_n}) = a_I \widetilde{h}_I \otimes x_{k_n}+ \text{a small error}$$ 
where $ a_{I }$ is a  bounded sequence of scalars and $\widetilde{h}_I$  spans  a copy of
$L_p $. 

In the mixed norm space $L_p (L_ q)  $ where $ 1< q < \infty $ and $ 1< p < \infty $
the bi-parameter Haar system forms an unconditional basis. Displaying   extraordinary
combinatorial strength, M. Capon \cite{CAPLL} exhibited  a  so called local product block basis
$ k_{I \times J }$, spanning a complemented copy of $L_p (L_q ) , $ such that
$$T k_{I \times J}  = a_{I \times J} k_{I \times J}+  \text{a  small error}.$$

\subsection{The present paper.}
Now we  turn to  describing  the main  ideas  in the approach of the present
paper.

Introducing a  transitive relation between operators $S,T$  on a Banach space $X ,$
we say that $T$ is a projectional factor of $S$ if there exist transfer operators $ A, B \colon X \to X $
such that 
\begin{equation}\label{14-1-21-5}
  S = ATB \quad\text{and}\quad   B  A = Id_X .
\end{equation}
If  merely $S = ATB$,  without the additional constraint $BA =Id_X$,  we say that $T$ is a factor of $S$, or equivalently that $S$ factors through $T$.

Clearly, if 
$T$ is a projectional factor of $S$ and $S$ one  of $R$ then
$T$ is a projectional factor of $R$, i.e., being a projectional factor is a transitive relation.
Given any operator $T : L_1(L_p) \to L_1(L_p)$ the goal is to show that either $T$ or $Id -T$ is a  factor of the  identity $Id : L_1(L_p) \to L_1(L_p)$.
In section~\ref{prelim-factors1} we    expand on the   quantitative aspects  of the transitive  relation \eqref{14-1-21-5} and the role it plays in
providing a step-by-step reduction of the problem, allowing for the replacement of  a given operator with a simpler one, that is easier to work with.

Let  $T : L_1(L_p) \to L_1(L_p)$ be a bounded linear operator. It is represented by a matrix $T = ( T^{I,J})$ of operators $ T^{I,J} : L_1 \to L_1 $,
indexed by pairs of dyadic intervals $(I,J) , $  that is,  on 
$f \in L_1(L_p)$ with Haar expansion
\begin{equation}
  f= \sum x_J h_J/|J|^{1/p}, \qquad x_J \in L_1,  
\end{equation}
the operator $T$ acts by 
\begin{equation}
 T f= \sum_I  (\sum_J T^{I,J}x_J ) h_I/|I|^{1/p}. 
\end{equation}
Theorem~\ref{main theorem}, the main result of this paper, 
asserts 
that there exists a bounded operator $ T^{0} :L_1 \to L_1 $
such that
\begin{equation*}\label{14-1-21-6}
  \text{$T$ is a projectional factor of
    $ T^0\otimes Id_{L_p}$
    ,}
\end{equation*}
meaning that  
there exist
bounded transfer operators $A, B : L_1(L_p) \to L_1(L_p) $ such that  $ B  A = Id _{L_1(L_p)} $ and 
\begin{equation}\label{14-1-21-6a}
  \vcxymatrix{ L_1(L_p) \ar[d]_{T}   &   L_1(L_p)\ar[l] ^{A} \ar[d]^{T^0\otimes Id} \\
L_1(L_p)   \ar[r] _{B} & L_1(L_p) }
\end{equation}
The  ideas involved in the proof of Theorem~\ref{main theorem},
are  based on the interplay of topological, geometric, and probabilistic principles.
Specifically we  build on compact families of $L_1 $-operators, extracted  from $  span \{  T^{I,J} \} $,  and large deviation estimates for  empirical processes: 

\begin{enumerate}[leftmargin=20pt,label=(\alph*)]

\item (Compactness.) We utilize  the  Semenov-Uksusov characterization  \cite{semenov:uksusov:2012} of
Haar multipliers on  $L_1  $ and uncover  compactness properties of the operators $ T^{I,J} : L_1 \to L_1 $.
See Theorem~\ref{relatively compact orbit} and Theorem~\ref{fundamental lemma}.
  
\item (Stabilization.) Large deviation  estimates for the empirical distribution method  gave rise to  a novel connection  between factorization problems on $L_ 1 (L_p) $ and
the concentration of measure phenomenon. See Lemma~\ref{concentration splitting} and Lemma~\ref{concentration splitting vector valued}.  
  . 

\end{enumerate}

\noindent
 {\bf Step 1.}
 We say that $T$ is a diagonal operator if $ T^{I,J} = 0$  for $ I \neq J , $ in which case we put  $  T^{L} =T^{L,L}$. 
The first step provides the reduction  to  diagonal operators.
Specifically,  Theorem~\ref{arbitrary to X-diagonal} asserts,
that for any operator $T = ( T^{I,J})$ there exists a diagonal operator $T_{\rm diag} =( T^{L} )$
such that 
\begin{equation}\label{14-1-21-6}
\text{$T$ is a projectional factor of $T_{\rm diag} =( T^{L} )$.}
\end{equation}
 The reduction \eqref{14-1-21-6}  results from compactness properties
 for the   family of $L_1 $ operators $ T^{I,J}$ established in Theorem~\ref{relatively compact orbit} and Theorem~\ref{fundamental lemma}.
 Specifically, if  $ f \in L_1$
then the set
\begin{equation}
  \{ T^{I,J}f \colon I, J \in  \mathcal{D} \} \subset L_1
\quad\text{is weakly relatively compact}; 
\end{equation}
if, moreover,  $T^{I,J}$ satisfies uniform off-diagonal estimates 
\begin{equation}
  \label{16-1-21-1}
\sup_{I,J}|\langle  T^{I,J}h_L ,  h_M \rangle| < \varepsilon_{L, M} , \quad\text{for}\quad  L \ne M .
  \end{equation}
  then, for $ \eta > 0 ,$
  there exists a stopping time collection of dyadic intervals $  \mathcal{A}$
  satisfying  $|\limsup  \mathcal{A}| > 1 -\eta$
 such that the set of operators 
\begin{equation}
\{ T^{I,J}P_{\mathcal{A}}  \colon  I, J \in \mathcal{D}    \} \subset L(L_1) 
\quad \text{is  relatively  norm-compact}. 
  \end{equation}
Recall    that   $   \mathcal{A} \subseteq  \mathcal{D}$ is a stopping time collection if
  for $K, L \in  \mathcal{A} $ and $ J \in  \mathcal{D} $ the assumption 
  $ K \subset J \subset L $ implies that  $ J \in  \mathcal{A} .$   
  By Theorem~\ref{diagonal formula}, 
  the orthogonal projection
  $$P_{\mathcal{A}} ( f ) = \sum_{I \in \mathcal{A}} \langle f , h_I\rangle h_I /|I| ,$$
  is bounded on $L_1 $ when  $   \mathcal{A} $ is a stopping time collection of dyadic intervals.

\par  
\noindent
 {\bf Step 2.}
Next we show that it suffices  to prove the factorization \eqref{14-1-21-6a} for   diagonal operators satisfying uniform off-diagonal estimates.  
We say that  $   T = (R^{L}) $ is a reduced diagonal operator if the  $R^L : L_1 \to L_1 $ satisfy  
\begin{equation}
  \label{14-1-21-11}
\sup_{L}|\langle R^Lh_I ,  h_J \rangle| < \varepsilon_{I, J} , \quad\text{for}\quad  I \ne J .
  \end{equation}
Proposition~\ref{entries uniformly close to multipliers} asserts that,
there exists a reduced diagonal operator $   T^{\rm red}_{\rm diag} = (R^{L}) $ satisfying \eqref{14-1-21-11},
such that
\begin{equation}\label{14-1-21-7}
\text{$T_{\rm diag} =( T^{L} )$ is a projectional factor of $   T^{\rm red}_{\rm diag} = (R^{L}) .$}
               \end{equation}
               To  prove \eqref{14-1-21-7} we  utilize the  compactness properties of  $T_{\rm diag} = (T^{L}) $
               together with
measure concentration estimates  \cite{BLM, Sch} associated to  the empirical distribution method.
See Lemma~\ref{concentration splitting} and Lemma~\ref{concentration splitting vector valued}. 

  \noindent
 {\bf Step 3.}
 Next we show that we may replace  reduced diagonal operators 
 by stable   diagonal operators.
  We say that $   T^{\rm  stbl}_{\rm diag} = (S^{L}) $ is a stable  diagonal operator  if
\begin{equation}
  \label{14-1-21-12}
  \| S^L - S^M \| < \varepsilon_M , 
  \end{equation}
  for dyadic intervals $M, L $ satisfying $ L \subseteq M . $  We obtain  in Proposition~\ref{compact entries to near-tensor}, that for any reduced diagonal operator $   T^{\rm red}_{\rm diag} $
  there exists a stable  diagonal operator $   T^{\rm  stbl}_{\rm diag} $ such that
\begin{equation}\label{14-1-21-8}
\text{$   T^{\rm red}_{\rm diag} = (R^{L}) $ is a projectional factor of  $   T^{\rm  stbl}_{\rm diag}  = (S^{L}) .$}
\end{equation}
We verify \eqref{14-1-21-8} exploiting again the compactness properties of  $T^{\rm red}_{\rm diag} = (R^{L}) $ in tandem with  
the probabilistic estimates  of  Lemma~\ref{concentration splitting} and Lemma~\ref{concentration splitting vector valued}. 

\noindent
{\bf Step 4.}
Proposition~\ref{T tensor identity} provides the final step of the argument. It  asserts that for any stable  diagonal operator $   T^{\rm  stbl}_{\rm diag}  =(S^L)$ there exists 
a bounded operator $ T^{0} :L_1 \to L_1 $ such that,
\begin{equation}\label{14-1-21-9}
\text{$   T^{\rm  stbl}_{\rm diag} $  is a projectional factor of $ T^0\otimes Id_X .$}
  \end{equation}
  To prove  \eqref{14-1-21-9}
  we set up a telescoping chain of operators connecting any of the $ S^L $ to  $S^{[0,1]} $ and invoke  the stability estimates \eqref{14-1-21-12} available for   the operators $S^I $
  when $ L \subset  I \subset  [0,1] . $ Thus we may finally take  $T^0 = S^{[0,1]} .$

  \noindent
{\bf Step 5.}
{Retracing} our steps, taking into account that the notion of projectional factors forms a  transitive relation, 
yields \eqref{14-1-21-6a}.

\section{Preliminaries}\label{prelim-factors}

\subsection{Factors and Projectional Factors up to Approximation} \label{prelim-factors1}
A common strategy in proving primariness of spaces such as $L_p$ is to study the behavior of a bounded linear operator on a $\sigma$-subalgebra on a subset of $[0,1)$ of positive measure. This process may have to be repeated several times.  We introduce some language that will make this process notationally easier.

\begin{dfn}\label{projectional factor}
Let $X$ be a Banach space, $T,S:X\to X$ be bounded linear operators and let $C\geq 1$, $\e\geq0$.
\begin{enumerate}[leftmargin=20pt,label=(\alph*)]

\item\label{projectional factor a} We say that {\em $T$ is a $C$-factor of $S$ with error $\e$} if there exist $A,B:X\to X$ with $\|BTA-S\| \leq \e$ and $\|A\|\|B\|\leq C$. We may also say that {\em $S$ $C$-factors through $T$ with error $\e$}.

\item\label{projectional factor b} We say that {\em $T$ is a $C$-projectional factor of $S$ with error $\e$} if there exists a complemented subspace $Y$ of $X$ that is isomorphic to $X$ with associated projection and isomorphism $P,A:X\to Y$ (\ie,   $A^{-1} P A$ is the identity on $X$),   so that $\|A^{-1}PTA - S\| \leq \e$ and $\|A\|\|A^{-1}P\|\leq C$. We may also say that {\em $S$ $C$-projectionally factors through $T$ with error $\e$}.

\end{enumerate}
When the error is $\e = 0$ we will simply say that $T$ is a $C$-factor or $C$-projectional factor of $S$.
\end{dfn}

\begin{rem}
If $T$ is a $C$-projectional factor of $S$ with error $\e$ then $I-T$ is a $C$-projectional factor of $I-S$ with error $\e$. Indeed, if $P$ and $A$ witness Definition \ref{projectional factor b}, then $PA = A$ and therefore $A^{-1}P(I-T)A = I - A^{-1}PTA$, i.e., $\|A^{-1}P(I-T)A - (I-S)\| \leq \e$.
\end{rem}

In a certain sense, being an approximate factor or projectional factor is a transitive property.

\begin{prop}
\label{approximate factor transitive}
Let $X$ be a Banach space and $R,S,T:X\to X$ be bounded linear operators.
\begin{enumerate}[leftmargin=20pt,label=(\alph*)]

\item \label{approximate factor transitive a} If $T$ is a $C$-factor of $S$ with error $\e$ and $S$ is a $D$-factor of $R$ with error $\de$ then $T$ is a $CD$-factor of $R$ with error $D\e +\de$.

\item \label{approximate factor transitive b} If $T$ is a $C$-projectional factor of $S$ with error $\e$ and $S$ is a $D$-projectional factor of $R$ with error $\de$ then $T$ is a $CD$-projectional factor or $R$ with error $D\e +\de$.

\end{enumerate}

\end{prop}

\begin{proof}
The first statement is straightforward and thus we only provide a proof of the second one.  Let $Y$ and $Z$  be complemented subspaces of $X$ which are isomorphic to $X$.
Let $P:X\to Y$ and $Q:X\to Z$,  be  the associated projections, and $A:X\to Y$ and $B:X\to Z$ the associated isomorphisms satisfying $\|A\|\|A^{-1}P\| \leq C$, $\|B\|\|B^{-1}Q\|\leq D$.
$\|A^{-1}PT - S\| \leq \e$ and $\|B^{-1}QSB - R\| \leq \de$.

We define $\tilde P = AQA^{-1}P$ and $\tilde A = AB$. Then, $\tilde P$ is a projection onto $\tilde A[X]$ and $\|\tilde P\|\|\tilde A^{-1}P\| \leq CD$. We obtain
\begin{equation*}
\|B^{-1}Q(A^{-1}PTA)B - B^{-1}QSB\| \leq \|B^{-1} Q\|\|B\|\|A^{-1}PTA - S\|\leq D \e
 \end{equation*}
 and thus $\|B^{-1}QA^{-1}PTAB - R\| \leq D\e +\de$.
Finally, observe that $$\tilde A^{-1}\tilde P = B^{-1}A^{-1}AQA^{-1}P = B^{-1}QA^{-1}P$$ and thus $\|\tilde A^{-1}\tilde PT\tilde A - R\|\leq D\e + \de$.
\end{proof}

The following explains the relation between primariness and approximate projectional factors.

\begin{prop}\label{primarity}
Let $X$ be a Banach space that satisfies Pe\l czy\'nsky's accordion property, i.e., for some $1\leq p\leq\infty$ we have that $X\simeq \ell_p(X)$. Assume that there exist $C\geq 1$ and  $0<\e<1/2$ so that every bounded linear operator $T:X\to X$ is a $C$-projectional factor with error $\e$ of a scalar operator, 
\ie, a  scalar multiple of the identity. Then, for every bounded linear operator $T:X\to X$ the identity $2C/(1-2\e)$ factors through either $T$ or $I-T$. In particular, $X$ is primary.
\end{prop}

\begin{proof}
Let $Y$ be a subspace of $X$ that is isomorphic to $X$ and complemented in $X$, with associated projection and isomorphism $P,A:X\to Y$,
 so that $\|A^{-1}P\|\|A\|\leq C$ and so that there exits a scalar $\la$ with $\|(A^{-1}P)TA - \la I\| \leq \e$. If $|\la| \geq 1/2$ then
\[\big\|\underbrace{\la^{-1}A^{-1}PTA}_{=:B} - I\big\| \leq 2\e <1\]
and thus $B^{-1}$ exists with  $\|B^{-1}\| \leq 1/(1-2\e)$. We obtain that if $S = B^{-1}\la^{-1}A^{-1}P$ then $STA = I$ and $\|S\|\|A\| \leq 2C/(1-2\e)$. If, on the other hand $|\la| <1/2$ then, because $\|A^{-1}P(I-T)A - (1-\la)I\|\leq\e$, we achieve the same conclusion for $I-T$ instead of $T$.

If $X = Y\oplus Z$ and $Q:X\to Y$ is a projection then we deduce that either $Y$ or $Z$ contains a complemented subspace isomorphic to $X$.  To see that we can assume
 that for some scalar $\lambda$, with $|\lambda|\ge 1/2$, $Q$ is a $C$-projectional factor with error $\e\in(0,1/2)$ of $\lambda I$. Otherwise  we replace $Q$ by $I-Q$.
 From what we proved so far we deduce that there  
 are operators  $S,A :X\to X$ so that $SQA = I$. 
 Then $W = QA(X)$ is a subspace of $Y$ that is isomorphic to $X$. It is also complemented via the projection $R = (S|_W)^{-1}S:X\to W$. 
So we obtain that $Y$ is a complemented subspace of $X$ and $X$ is isomorphic to complemented subspace of $Y$. Since in addition  $X$ satisfies 
 the accordion property it follows  from Pe\l czy\'nsky's famous classical argument from \cite{pelczynski:1960}  that $X\simeq Y$. 
 Similarly, if $(I-Q)$ is a factor of the identity  we deduce $X\simeq Z$.
\end{proof}

\subsection{The Haar system in $L_1$}
We denote by $L_1$ the space of all (equivalence classes of) integrable scalar functions $f$ with domain $[0,1)$ endowed with the norm $\|f\|_1 = \int _0^1|f(s)| ds$. We will denote the Lebesgue measure of a measurable subset $A$ of $[0,1)$ by $|A|$.

We denote by $\mathcal{D}$ the collection of all dyadic intervals in $[0,1)$, namely
\begin{equation*}
  \mathcal{D} = \Big\{\Big[\frac{i-1}{2^j},\frac{i}{2^j}\Big): j\in\mathbb{N}\cup\{0\}, 1\leq
  i\leq 2^j\Big\}.
\end{equation*}
We define the bijective function
$\iota: \mathcal{D}\to \{2,3,\ldots\}$ by
\begin{equation*}
  \Big[\frac{i-1}{2^j},\frac{i}{2^j}\Big)
  \stackrel{\iota}{\mapsto} 2^j + i .
\end{equation*}
The function $\iota $ defines a linear order on $\mathcal{D}$.  We recall the definition of the Haar system
$(h_I)_{I\in\mathcal{D}}$. For $I = [(i-1)/2^j,i/2^j)\in \mathcal{D}$ we define $I^+, I^-\in\mathcal{D}$ as follows:
$I^+ = [(i-1)/2^j,(2i-1)/2^{j+1})$, $I^- = [(2i-1)/2^{j+1},i/2^j)$, and
\begin{equation*}
  h_I
  = \chi_{I^+} - \chi_{I^-}.
\end{equation*}
We additionally define $h_\emptyset = \chi_{[0,1)}$ and
$\mathcal{D}^+ = \mathcal{D}\cup\{\emptyset\}$. We also define $\iota(\emptyset) = 1$. Then, $(h_I)_{I\in\mathcal{D}^+}$ is a
monotone Schauder basis of $L_1$, with the linear order induced by $\iota$. Henceforth, whenever we write $\sum_{I\in\mathcal{D}^+}$ we will always mean that the sum is taken with this linear order $\iota$.

For each $n\in\mathbb{N}\cup\{0\}$ we define
\[\mathcal{D}_n = \{I\in\mathcal{D}: |I| = 2^{-n}\}\text{ and }\mathcal{D}^n = \{\emptyset\}\cup(\cup_{k=0}^n\mathcal{D}_k).\]

An important realization, that will be used multiple times in the sequel is the following. Let $I\in\mathcal{D}$.  Then there exists a unique $k_0 \in \N$ and a  unique decreasing sequence of 
intervals  $(I_k)_{k=0}^{k_0}$ in  $(\mathcal{D}^+)$, so that $I_0 = \emptyset$, $I_1 = [0,1)$, and $I_{k_0}=I$, and for $k=1,2, \ldots, k_{0-1}$, $I_{k+1}=I^+_k$, or $I_{k+1}=I^-_k$.
In other words  $(I_k)_{k=1}^{k_0}$ consists of all elements of $\mathcal D^+$ which contain $I$, decreasingly ordered.
  For $k=1,2, \ldots, k_0-1$   put $\theta_k = 1$, if $I_{k+1} = I_k^+$ and $\theta_k = -1$ if $I_{k+1} = I_k^-$. We then have the following formula, already discovered by Haar,
\begin{equation}
\label{normalized intervals}
|I|^{-1}\chi_I  = |I_{k_0}|^{-1}\chi_{I_{k_0}} = h_{I_0} + \sum_{k=1}^{k_0-1}\theta_k|I_k|^{-1}h_{I_k}.
\end{equation}
Note that in the above representation, if we define $I_{k_0} = I$, then $I_k = I_{k-1}^-$ or $I_k = I_{k-1}^+$ for $k=2,\ldots,k_0$. To simplify notation, we will henceforth make the convention $\theta_0 = 1$ and $|I_0|^{-1} = |\emptyset|^{-1}  = 1$ to be able to write
\begin{equation}
\label{normalized intervals simple}
|I_{k_0}|^{-1}\chi_{I_{k_0}} = \sum_{k=0}^{k_0-1}\theta_k|I_k|^{-1}h_{I_k}.
\end{equation}
This representation will be used multiple times in this paper.

A relevant definition is that of $[\mathcal{D}^+]$, the collection of all sequences $(I_k)_{k=0}^\infty$ in $\mathcal{D}^+$ so that $I_0 = \emptyset$, $I_1 = [0,1)$, and for each $k\in\N$, $I_{k+1} = I_k^+$ or $I_{k+1} = I_k^-$.
Note that for $(I_k)_{k=0}^\infty\in[\mathcal{D}^+]$ and $k\in\N$, $I_k\in\mathcal{D}_{k-1}$. 
 Each $(I_k)_{k=0}^\infty$ defines a sequence $(\theta_k)_{k=1}^\infty$ as described in the paragraph above. 
This yields a bijection between $[\mathcal{D}^+]$ and $\{-1,1\}^\N$. This fact will be used more than once. On $\{-1,1\}^\N$ we will consider the product of the uniform distribution on $\{-1,1\}$, which
via this bijection generates a probability  on $[\mathcal{D}^+]$, which we will also denote by $|\cdot|$. Also, we consider on $[\cD^+] $ the  image topology
 of the product of the discrete topology on $\{-1,1\}$ via that bijection.

\subsection{Haar multipliers on $L_1$}
A Haar multiplier is a linear map $D$, defined on the linear span of the Haar system, for which every Haar vector $h_I$ is an eigenvector with eigenvalue $a_I$. We denote the space of bounded Haar multipliers $D:L_1\to L_1$ by $\mathcal{L}_{HM}(L_1)$. In this subsection we recall a formula for the norm of a Haar multiplier that was observed by Semenov and Uksusov in \cite{semenov:uksusov:2012}. We then use Haar multipliers to sketch a proof of the fact that every bounded linear operator on $L_1$ is an approximate 1-projectional factor of a scalar operator.

\begin{prop}
\label{branches variation}
Let $(I_k)_{k=0}^\infty\in[\mathcal{D}^+]$ associated to $(\theta_k)_{k=1}^\infty\in\{-1,1\}^\N$. For $k\in\N$ define $B_k =  I_k\setminus I_{k+1}$
and let $(a_k)_{k=0}^n$ be a sequence of scalars.

Then we have
\begin{align}
\label{branches variation 1}
                      &\left\|\sum_{k=0}^na_k\theta_k|I_k|^{-1} h_{I_k}\right\|_{L_1}  \leq \sum_{k=1}^n|a_k - a_{k-1}| + |a_n|.\\
  \intertext{and for any        $1\le m< n$}             
\label{branches variation 1a}
                      & \left\|\Big(\sum_{k=0}^na_k\theta_k|I_k|^{-1} h_{I_k}\Big)\big|_{\bigcup_{j=m}^n B_j}\right\|_{L_1}\ge \frac{1}{3}\Big(\sum_{k=m+1}^n|a_k - a_{k-1}|+|a_n|\Big) .
\end{align}
\end{prop}

\begin{proof} Note that the sequence  $(B_k)_{k=1}^\infty $ is a partition of $[0,1)$ and for $k\in\N$
  $B_k$  is the set in $[0,1]$ of measure $2^{-k}$, on which $\theta_k h_{I_k}$ takes the value $-1$. 
  Let $f = a_0 h_\emptyset + \sum_{k=1}^n \theta_ka_k|I_k|^{-1}h_{I_k}$. For $k\in\N$ put $b_k = a_k$ if $k\leq n$ and $b_k = 0$ otherwise. For each $k\in\N$ the function $f$ is constant on $B_k$ and in fact for $s\in B_k$ we have
\begin{equation*}
f(s) =b_0 + \sum_{j=1}^{k-1}|I_j|^{-1}b_j - |I_k|^{-1}b_k = b_0 + \sum_{j=1}^{k-1}2^{j-1}b_j - 2^{k-1}b_k =: c_k.
\end{equation*}
Therefore, for any $m=1,2,\ldots n$
\begin{equation}
\begin{split}
\left\|f \chi_{\bigcup_{j=m}^n B_j} \right\|_{L_1} = \sum_{k=m}^\infty |X_k|\label{branches variation 2}
\end{split}
\end{equation}
where for each $k\in\N$,
\[X_k = \frac{c_k}{2^{k}} = \frac{1}{2^{k}}b_{0} + \sum_{j=1}^{k-1}\frac{2^{j-1}}{2^{k}}b_j - \frac{1}{2}b_k.\]
Putting $X_0 = 0$, a calculation yields that for all $k\in\N$
\begin{equation}
\label{branches variation 3}
X_k = \frac{1}{2}X_{k-1} + \frac{1}{2}\left(b_{k-1} - b_k\right). 
\end{equation}
Applying the triangle inequality to \eqref{branches variation 2} and \eqref{branches variation 3} we conclude

\begin{align*}
\| f\|_{L_1}&=\sum_{k=m}^\infty |X_k|=\sum_{k=1}^\infty 2|X_k|-|X_{k-1}|\\
                 &\le \sum_{k=1}^\infty |2X_k -X_{k-1}|=\sum_{k=1}^\infty|b_k-b_{k-1}|
\end{align*}
which yields \eqref{branches variation 1}. In order to obtain \eqref{branches variation 1a}, we deduce from \eqref{branches variation 3}
\begin{align*}
&\sum_{k=m+1}^\infty |X_k|
                                                           \ge\frac12 \sum_{k=m+1}^\infty |b_k-b_{k-1}|-\frac12\sum_{k=m} |X_k| 
                                                             \intertext{and therefore}
          & \frac32   \sum_{k=m+1}^\infty |X_k| +\frac12 |X_{m}|      \ge  \frac12 \sum_{k=m+1}^\infty |b_k-b_{k-1}|                                        
   \intertext{which yields}
  & \| f \chi_{\bigcup_{j=m}^n B_j}\|_{L_1}  =      \sum_{k=m}^\infty |X_k|  \ge     \sum_{k=m+1}^\infty |X_k| +\frac13|X_m|       \ge \frac13  \sum_{k=m+1}^\infty |b_k-b_{k-1}|                                                                                                  
\end{align*}
and proves \eqref{branches variation 1a}.
\end{proof}

\begin{thm}[Semenov-Uksusov, \cite{semenov:uksusov:2012}]
\label{diagonal formula}
Let $(a_I)_{I\in\mathcal{D}^+}$ be a collection of scalars, and $D$ be the associated Haar multiplier. Define
\begin{equation}
\label{diagonal triple norm}
\trn{D} = \sup\Big(\sum_{k=1}^\infty\big |a_{I_k} - a_{I_{k-1}}\big|+\lim_k\big|a_{I_k}\big|\Big)
\end{equation}
where the supremum is taken over all $(I_k)_{k=0}^\infty\in[\mathcal{D}^+]$. Then, $D$ is bounded (and thus extends to a bounded linear operator on $L_1(X)$) if and only if $\trn{D}<\infty$. More precisely,
\begin{equation}
\label{diagonal equivalent norm}
\|D\| \leq \trn{D}\leq 3\|D\|.
\end{equation}
\end{thm}

\begin{proof}
By \eqref{normalized intervals}, $D$ is  always well defined on the linear span of the set $\mathcal{X} = \{|I|^{-1}\chi_I:I\in\mathcal{D}\}$. In fact, the closed convex symmetric hull of $\mathcal{X}$ is the unit ball of $L_1$. We deduce that $\|D\| = \sup\{\|Df\|:f\in\mathcal{X}\}$, under the convention that $\|D\| = \infty$ if and only if $D$ is unbounded. Fix $f = |I|^{-1}\chi_I\in\mathcal{X}$. Use \eqref{normalized intervals} to write
\[f = |I_{k_0}|^{-1}\chi_{I_{k_0}} = \sum_{k=0}^{k_0-1}\theta_k|I_k|^{-1}h_{I_k},\text{ \ie, }Df =  \sum_{k=0}^{k_0-1}a_k\theta_k|I_k|^{-1}h_{I_k}.\]
Extend $(I_k)_{k=0}^{k_0}$ to a branch $(I_k)_{k=0}^\infty$. By \eqref{branches variation 1} we have
\begin{equation}
\label{diagonal formula eq1}
\frac{1}{3}\big(\sum_{k=1}^{k_0-1}|a_{I_k} - a_{I_{k-1}}| + |a_{I_{k_0-1}}|\big) \leq \|Df\|_{L_1} \leq \sum_{k=1}^{k_0-1}|a_{I_k} - a_{I_{k-1}}| + |a_{I_{k_0-1}}|.
\end{equation}
By the triangle inequality, $\|Df\|_{L_1} \leq \sum_{k=1}^{\infty}|a_{I_k} - a_{I_{k-1}}| + \lim_k|a_{I_{k}}| \leq \trn{D}$. The lower bound is achieved by taking in \eqref{diagonal formula eq1} all $f\in\mathcal{X}$.
\end{proof}

The following special type of Haar multiplier will appear in the sequel.
\begin{exa}
Let $\mathscr{A}\subset[\mathcal{D}^+]$ be a non-empty set and define the set $\mathcal{A} = \cup_{k_0=0}^\infty\{I_{k_0}:(I_k)_{k=0}^\infty\in\mathscr{A}\}\subset\mathcal{D}^+$. Let $P_\mathscr{A}$ denote the Haar multiplier that has entries $a_I = 1$ for $I\in\mathcal{A}$ and $a_I = 0$ otherwise. Then, 
 by Theorem \ref{diagonal formula},
$\|P_\mathscr{A}\|\leq \trn{P_\mathscr{A}} = 1$ and therefore $P_\mathscr{A}$ defines a norm-one projection onto $Y_\mathscr{A} = \overline{\langle\{h_I:I\in\mathcal{A}\}\rangle}$.
\end{exa}

The following elementary remark will be useful eventually.
\begin{rem}
\label{restricted triple bar}
Let $\mathscr{A}$ be a non-empty closed subset of $[\mathcal{D}^+]$ and $\mathcal{A} = \cup_{k_0=0}^\infty\{I_{k_0}:(I_k)_{k=0}^\infty\in\mathscr{A}\}$. Let $D$ be a Haar multiplier with entries that are zero outside $\mathcal{A}$. Then, $\trn{D} = \sup_{(I_k)_{k=0}^\infty\in\mathscr{A}}(\sum_{k=1}^\infty|a_{I_k}-a_{k-1}| + \lim_k|a_{I_k}|)$.
\end{rem}

Haar multipliers provide a short path to a proof of the fact that every operator on $L_1$ is an approximate 1-projectional factor of a scalar operator, which in turn yields
Enflo's theorem~\cite{EM} that  $L_1$ is primary.

\begin{thm}
\label{icebreaker}
The following are true in the space $L_1$.
\begin{enumerate}[leftmargin=20pt,label=(\roman*)]

\item Let $D:L_1\to L_1$ be a bounded Haar multiplier. For every $\e>0$, $D$ is a 1-projectional factor with error $\e$ of a scalar operator.

\item Let $T:L_1\to L_1$ be a bounded linear operator. For every $\e>0$, $T$ is a 1-projectional factor with error $\e$ of a bounded Haar multiplier $D:L_1\to L_1$.

\end{enumerate}
In particular, for every $\e>0$, every bounded linear operator $T:L_1\to L_1$ is a 1-projectional factor with error $\e$ of a scalar operator.
\end{thm}

We wish to provide a sketch of the proof of the above. Firstly, we will use it at the end of the paper and secondly it provides an introduction to the basis of the methods used in the paper. Now, and numerous times in the sequel, we require the following notation and definition.

\begin{ntt}
For every disjoint collection $\De$ of $\mathcal{D}^+$ and $\theta\in\{-1,1\}^\De$ we denote $h_\De^\theta = \sum_{J\in\De}\theta_Jh_J$. If $\theta_J = 1$ for all $J\in\De$ we write  $h_\De = \sum_{j\in\De}h_J$. For a finite disjoint collection $\De$ of $\mathcal{D}$ we denote $\De^* = \cup\{I:I\in\De\}$.
\end{ntt}

\begin{dfn}\label{faithful Haar system}
A {\em faithful Haar system} is a collection $(\tilde h_I)_{I\in\mathcal{D}^+}$ so that for each $I\in\mathcal{D}^+$ the function $\tilde h_I$ is of the form $\tilde h_I = h_{\De_I}^{\theta_I}$,
 for some finite disjoint collection $\De_I$ of $\mathcal{D}$,
 and so that
\begin{enumerate}[leftmargin=23pt,label=(\roman*)]

\item $\De_\emptyset^* = \De_{[0,1)}^* = [0,1)$ and for each $I\in\mathcal{D}$ we have $|\De_I| = |I|$,

\item for every $I\in\mathcal{D}$ we have that $\De_{I^+}^* = [\tilde h_\emptyset\tilde h_I = 1]$ and $\De_{I^-}^* = [\tilde h_\emptyset\tilde h_I = -1]$.

\end{enumerate}
\end{dfn}

\begin{rem}
It is immediate that $(\tilde h_\emptyset \tilde h_I)_{I\in\mathcal{D}^+}$ is distributionally equivalent to $(h_I)_{I\in\mathcal{D}^+}$. Therefore, $(\tilde h_I)_{I\in\mathcal{D}^+}$ is isometrically equivalent to $(h_I)_{I\in\mathcal{D}^+}$, both in $L_1$ and in $L_\infty$. In particular,
\[Pf = \sum_{I\in\mathcal{D}^+}\langle \tilde h_I, f\rangle |I|^{-1}\tilde h_I\]
defines a norm-one projection onto a subspace $Z$ of $L_1$ that is isometrically isomorphic to $L_1$. Note that, unless $h_\emptyset = 1$, $P$ is not a conditional expectation as $P\chi_{[0,1)} = 0$. Instead, it is of the form $Pf = \tilde h_\emptyset E(\tilde h_\emptyset f|\Sigma)$,  where $\Sigma = \sigma(\tilde h_\emptyset \tilde h_I)_{I\in\mathcal{D}^+}$. Since $\tilde h_\emptyset$ is not $\Sigma$-measurable it cannot be eliminated. The advantage of the notion of a faithful Haar system is that one can be constructed in every tail of the Haar system. The drawback is that it causes a slight notational burden when having to adjust for the initial function $\tilde h_\emptyset$ in several situations.

We will several times recursively construct faithful Haar systems $(\tilde h_I)_{I\in\mathcal{D}^+}$, which means that we first choose $\tilde h_\emptyset$, secondly $\tilde h_{[0,1)}$,
and then $\tilde h_I$, $I\in \cD$,  assuming that $\tilde h_J$ was chosen for all $J\in\cD ^+$  with $\iota(J) < \iota(I)$.
\end{rem}

\begin{proof}[Proof of Theorem \ref{icebreaker}]
Let us sketch the proof of the first statement. Let $(a_I)_{I\in\mathcal{D}^+}$ be the entries of $D$. For every $I\in\mathcal{D}$ denote by $Q_I$ the Haar multiplier that has entries 1 for all $J\subset I$ and zero all others. Then, $\trn{Q_I} = 1$.
 First note that, for every $\e>0$, there exits $I_0\in\mathcal{D}^+$ so that $\trn{DQ_{I_0} - a_{I_0}Q_{I_0}}\leq \e$. Otherwise, we could  easily deduce $\trn{D} = \infty$. Construct a  dilated  and renormalized faithful Haar system $(\tilde h_I)_{I\in\mathcal{D}^+}$ with closed linear span $Z$ in the range of $Q_{I_0}$ and let $P:L_1\to Z$ be the corresponding norm-one projection and $A:L_1\to Z$ be an onto isometry. Then, $\|A^{-1}PDA - a_{I_0}I\| \leq \e$.

For the second part we will use that the Rademacher sequence $(r_n)_n$ (\ie, $r_n = \sum_{L\in\mathcal{D}_n}h_L$, for  $n\in\N$) ) is weakly null in $L_1$ and $w^*$-null in $(L_1)^* \equiv L_\infty$. Using this fact, we inductively construct a faithful Haar system $(\tilde h_I)_{I\in\mathcal{D}^+}$ so that for each $I\neq J$ we have
\[\big|\big\langle\tilde h_I,T\big(|J|^{-1}\tilde h_J\big)\big\rangle\big|\leq \e_{(I,J)},\]
where $(\e_{(I,J)})_{(I,J)\in\mathcal{D}^+}$ is a pre-chosen collection of positive real numbers with $\sum\e_{(I,J)} \leq \e$. This is done as follows. If we have chosen $\tilde h_I$ for $\iota(I) = 1,\ldots,k-1$. Let $I\in\mathcal{D}^+$ with $\iota(I) = k$ and let $I_0$ be the predecessor of $I$, i.e., either $I = I_0^+$ or $I = I_0^-$. Let us assume $I = I_0^+$. We then choose the next function $\tilde h_I$ among the terms of a Rademacher sequence with support $[\tilde h_\emptyset\tilde h_{I_0} = 1]$. Denote by $Z$ the closed linear span of $(\tilde h_I)_{I\in\mathcal{D}^+}$ and take the canonical projection $P:L_1\to Z$ as well as the onto isometry $A:L_1\to Z$ given by $A h_I = \tilde h_I$. Consider the operator $S=A^{-1}PTA:L_1\to L_1$ and note that for all $I\neq J$ we have $|\langle h_I, S(|J|^{-1}h_J) \rangle| = |\langle\tilde h_I, T(|J|^{-1}\tilde h_J )\rangle| \leq \e_{(I,J)}$. It follows that the entries $a_I = \langle h_I, S(|I|^{-1}h_I )\rangle$ define a bounded Haar multiplier $D$ and $\|S - D\| \leq \e$, i.e., $T$ is a 1-projectional factor with error $\e$ of $D$.
\end{proof}

\subsection{ 
Haar system spaces}
We define 
Haar system spaces. These are Banach spaces of scalar function generated by the Haar system in which two functions with the same distribution have the same norm. This abstraction does not impose any notational burden to the proof of the main result. The only difference to the case $X= L_p$ is the normalization of the Haar basis. Properties such as unconditionality of  the Haar system or reflexivity of $L_p$ are never deployed.

\begin{dfn}
\label{RI def}

A {\em Haar system space} $X$ is the completion of  $Z = \langle\{ h_L:L\in\mathcal{D}^+\}\rangle = \langle\{\chi_I:I\in\mathcal{D}\}\rangle$ under a norm $\|\cdot\|$ that satisfies the following properties.
\begin{enumerate}[leftmargin=20pt,label=(\roman*)]

\item\label{RI def 1} If $f$, $g$ are in $Z$ and $|f|$, $|g|$ have the same distribution then $\|f\| = \|g\|$. 

\item\label{RI def 2} $\|\chi_{[0,1)}\| = 1$.

\end{enumerate}
We denote the class of Haar system spaces by $\cH$.

\end{dfn}
Obviously, property \ref{RI def 2} may be achieved by scaling the norm of a space that satisfies \ref{RI def 1}. We include it anyway for notational convenience.

An important class of spaces which satisfy  Definition \ref{RI def}, according to \cite[Proposition 2.c.1]{lindenstrauss:tzafriri:1979:partII}, are separable  rearrangement invariant
function spaces on $[0,1]$. Recall that  a (non-zero) Banach space $Y$ of measurable scalar functions on $[0,1)$ is called {\em rearrangement invariant}
(or as in  \cite{rodin:semyonov:1975} {\em symmetric})  if the following conditions hold true: First,  whenever $f\in Y$ and $g$ is a measurable function with $|g|\leq |f|$ a.e.  then $g\in Y$ and $\|g\|_Y \leq \|f\|_Y$. Second,  if $u,v$ are in $Y$ and they have the same distribution then $\|u\|_Y = \|v\|_Y$.

The following properties of  a Haar  system space $X$ follow from elementary arguments. For completeness, we provide the proofs.

\begin{prop} Let $X$ be a Haar system space.
\label{RI properties}
\begin{enumerate}[leftmargin=21pt,label=(\alph*)]

\item\label{RI properties 1} For every $f\in Z = \langle\{\chi_I:I\in\mathcal{D}\}\rangle$ we have $\|f\|_{L_1}\leq \|f\|\leq \|f\|_{L_\infty}$. Therefore, $X$ can be naturally identified with a space of measurable scalar functions on $[0,1)$ and ${\overline Z^{\|\cdot\|_{L_\infty}}} \subset X \subset{L_1}$.

\item \label{RI properties 2}$Z = \langle\{\chi_I:I\in\mathcal{D}\}\rangle$ naturally coincides with a subspace of $X^*$ and its closure $\overline Z$ in $X^*$ is also  a Haar  system space.

\item \label{RI properties 3} The Haar system, in the usual linear order, is a monotone Schauder basis of $X$.

\item \label{RI properties 4} For a finite union $A$ of elements of $\mathcal{D}$ we put $\mu_A = \|\chi_A\|^{-1}_X$ and $\nu_A = \|\chi_A\|^{-1}_{X^*}$. Then, $\mu_A\nu_A = |A|^{-1}$. In particular, $(\nu_Lh_L,\mu_Lh_L)_{L\in\mathcal{D}^+}$ is a biorthogonal system in $X^*\times X$.

\item  \label{RI properties 5} A faithful Haar system $(\widehat h_L)_{L\in\mathcal{D}^+}$ is isometrically equivalent to $(h_L)_{L\in\mathcal{D}^+}$. In particular, $Pf = \sum_{L\in\mathcal{D}^+}\langle \nu_L\widehat h_L,f \rangle\mu_L\widehat h_L$ defines a norm-one projection onto a subspace of $X$ that is isometrically isomorphic to $X$.

\end{enumerate}
\end{prop}
\begin{proof}\! By the first  condition in Definition~\ref{RI def},  we have 
$$\big\|\sum_{I\in \cD_n} \!a_I\chi_{\pi(I)} \big\|\!=\!\big\|\sum_{I\in \cD_n} \!a_I\chi_{I} \big\| ,$$
for all $n\in\N$, all permutations $\pi$ on $\cD_n$, and all scalar families $(a_I:I\in \cD_n)$.
 
To show the first inequality  in  \ref{RI properties 1} let  $n\in\N$, $f=\sum_{I\in \cD_n} a_I \chi_I\in Z$ and let $\pi:\cD_n\to \cD_n $,
be cyclic (\ie, $\{\pi^r(I):r=1,2\ldots,2^n\}=\cD_n$ for $I\in \cD_n$).
 Then 
 $$\|f\|=\big\| |f|\big\| \ge \frac1{2^n} \Big\|\sum_{r=1}^{2^n}  \sum_{I\in \cD_n} |a_I| \chi_{\pi^r(I)}\Big\|=\frac1{2^n}\Big|\sum_{I\in \cD_n} |a_I|\Big|=\|f\|_{L_1}.$$
    The second inequality in  \ref{RI properties 1} follows from the observation that for each $n\in\N$ the family $(\chi_I: I \in \cD_n)$ is $1$-unconditional.
    
    We identify each $g\in Z$ with  the bounded  functional $x^*_g$, defined by $x^*_g(f)=\int_0^1 f g$, and we denote the dual norm by $\|\cdot\|_{*}$. From this representation it is clear that $\|\cdot\|_*$ also 
    satisfies the first  condition in Definition~\ref{RI def}.
    Since $\|1_{[0,1)]}\|=1$ and since for all $f\in Z$, $ \int f \le \|f\|_1\le \|f\|$, we deduce that the second condition in Definition~\ref{RI def} holds true  for  the norm $\|\cdot\|_*$. 
    
   Let $(h_n)$ be the Haar basis linearly  ordered   in the usual way, meaning that if $m<n$, then either $\supp(h_n)\subset \supp(h_m)$ or 
    $\supp(h_n)\cap \supp(h_m)=\emptyset$. The claim of condition \ref{RI properties 3}
    follows from the fact that if $f=\sum_{j=1}^n a_j h_j\in Z$, then for any scalar $a_{n+1}$ the absolute values of the functions 
   $f+a_{n+1} h_{n+1}$ and $f-a_{n+1} h_{n+1}$ have the same distribution and their average is $f$.
   
   Let $n\!\in\!\N$ and $I\!\in\!\cD_n$ using for $k\!>\!n$ cyclic permutations on $\{ J\!\in\!\cD_k, J\!\subset\!I\}$ we deduce that 
   $\sup_{f\in Z,\|f\|\le 1} \int_I f$ is attained   for $f=\chi_I/\|\chi_I\|$ and thus 
   $\|\chi_I\|\cdot\|\chi_I\|_*=2^{-n}$. Since secondly, for each $n$, $(\chi_I: I \in \cD_n)$ is an orthogonal family, we deduce \ref{RI properties 4}.
   
   Since faithful Haar systems  have the same joint  distribution  we deduce the first part of \ref{RI properties 5}. Since by   \ref{RI properties 2}, this is also true 
   with respect to the dual norm we deduce the second part of \ref{RI properties 5}.
\end{proof}

In different parts of proof we will require additional properties of Haar system spaces.
The following class of Haar  system spaces is the one for which we prove our main theorem.
\begin{dfn}
\label{nice RI class}
 $\cH^*$  is the class of all Banach spaces $X$ in $\cH$ satisfying
\begin{enumerate}[leftmargin=22pt,label=($\star$)]
\item\label{nice RI class 1} the Rademacher sequence $(r_n)_n$ is not equivalent  to the $\ell_1$-unit vector basis.
\end{enumerate}
 $\cH^{**}$  is the class of all Banach spaces $X$ in $\cH$ satisfying
\begin{enumerate}[leftmargin=22pt,label=($\star\star$)]\label{nice RI class,2}
\item\label{nice RI class 2} no subsequence of the $X$-normalized Haar system  $(\mu_Lh_L)_{L\in\mathcal{D}^+}$ is equivalent  to the $\ell_1$-unit vector basis.
\end{enumerate}
\end{dfn}
\begin{rem}   Examples of Haar system spaces which satisfy \ref{nice RI class 1} and \ref{nice RI class 2} are separable reflexive r.i. spaces.
 
 We note and will use several times that  \ref{nice RI class 1}  for Haar system spaces, is equivalent with the condition that  the Rademacher sequence $(r_n)$
 is weakly null. To see this, first note
   that for any $(a_n)\in c_{00}$, any  $\sigma=(\sigma_n)\subset\{\pm1\}$, and  
permutation $\pi$ on $\N$
the distribution
of $\sum_{n\in\N} a_n\sigma_n r_{\pi(n)}$, does not depend on $\sigma$ on $\pi$.  It follows that $(r_n)$ is a symmetric basic  sequence in $X$.
This implies that either $r_n$ is equivalent to the $\ell_1$ unit vector basis or it is weakly null in $X$. Indeed,
if it is not equivalent to the unit vector basis of $\ell_1$, and by symmetry no subsequence,  is equivalent to the  $\ell_1$ unit vector basis,  it must
by Rosenthal's $\ell_1$ Theorem  have  weakly Cauchy subsequence and thus 
for  some subsequence $(n_k)\subset\N$ the sequence  $(r_{n_{2k}}- r_{n_{2k-1}}:k\in\N)$  is weakly null. But then also the sequence
 $(r_{n_{2k}}+ r_{n_{2k-1}}:k\in\N)$  is weakly null, and thus $r_{n_{2k}}$ is weakly null and by symmetry $(r_n)$ is weakly null.

\end{rem}

\subsection{Complemented subspaces of ${L_1(X)}$ isomorphic to $L_1(X)$}

Let $E$, $F$ be Banach spaces. The projective tensor product of $E$ and $F$ is the completion of the algebraic tensor product  $E\otimes F$ under the norm
\begin{equation}
\label{projective tensor norm}
\|u\| = \inf\Big\{\sum_{n=1}^N\|e_n\|\|f_n\|: u = \sum_{n=1}^Ne_n\otimes f_n\Big\}.
\end{equation}
It is well known and follows from the definition of Bochner-Lebesque spaces that for any Banach space $X$, $L_1\otimes_\pi X\equiv L_1(X)$ via the identification $(f\otimes x)(s) = f(s)x$. Then,  $L_\infty(X^*)$ canonically embeds into $(L_1(X))^*$ via the identification $\langle u,v\rangle = \int_0^1 \langle u(s),v(s)\rangle ds$.  Recall that 
by the  definition of tensor norms   the projective tensor norm  satisfies the following  property we will use.
\begin{enumerate}[leftmargin=18pt,label=(\textbullet)]

\item\label{standard fare tensor a} For any pair of bounded linear operators $T:E\to E$ and $S:F\to F$ there exists a unique bounded linear operator $T\otimes S:E\otimes_\pi F\to E\otimes_\pi F$ with $(T\otimes S)(e\otimes f) = (Te)\otimes(Sf)$ and $\|T\otimes S\| = \|T\|\|S\|$.

\end{enumerate}

The next standard statement explains one of the main features of the projective tensor product. For the sake of completeness, and because it is essential in this paper, we include the proof.

\begin{prop}
\label{L_1(X) complemented in L_1(X)}
Let $Z$ be a subspace of $L_1$ that is isometrically isomorphic to $L_1$ via $A:L_1\to Z$ and 1-complemented in $L_1$ via $P:L_1\to Z$. Let $X$ be a Banach space and let $W$ be a subspace of $X$ that is isometrically isomorphic to $X$ via $B:X\to W$ and $1$-complemented in $X$ via $Q:X\to W$. 

Then the  space
$Z(W) =\overline{Z\otimes W}^{L_1(X)}$  coincides with  $Z\otimes_\pi W$ and 
is isometrically isomorphic to $L_1(X)$ via $A\otimes B:L_1(X)\to Z(W)$ and 1-complemented in $L_1(X)$ via $P\otimes Q:L_1(X)\to Z(W)$.
\end{prop}

\begin{proof}
It is immediate that $P\otimes Q$ is a norm-one projection onto $Z(W)$ and that $A\otimes B$ is a norm-one map with dense image. It also follows that $A\otimes B$ is 1-1 on $L_1\otimes X$. One way to see this is to identify $L_1\otimes X$ and $Z\otimes W$ with spaces of bilinear forms on $(L_1)^*\times X^*$ and $Z^*\times W^*$ respectively. To conclude that $A\otimes B$ is an isometry and that $Z(W)=Z\otimes_{\pi} W$ take $u$ in $L_1\otimes X$. Note that  $v:=(A\otimes B)(u)$ is in $Z\otimes W\subset L_1\otimes X$ and write $v = \sum_{i=1}^nf_i\otimes x_i$, where $f_1,\ldots,f_n\in L_1$ and $x_1,\ldots,x_n\in X$. We will see that $\sum_{i=1}^n\|f_i\|\|x_i\| \geq \|u\|$, which will imply the conclusion, by the definition of $\|v\|$. Indeed, $v = (P\otimes Q)(v) = \sum_{i=1}^n(Pf_i)\otimes(Qx_i)$ and
\begin{align*}
\|v\| &\geq \sum_{i=1}^n\|Pf_i\|\|Qx_i\| = \sum_{i=1}^n\|A^{-1}Pf_i\|\|B^{-1}Qx_i\|\\
& \geq \big\|\underbrace{\sum_{i=1}^n\big(A^{-1}Pf_i\big)\otimes\big(B^{-1}Qx_i\big)}_{=:y}\big\|.
\end{align*}
It is immediate that $(A\otimes B)(y) = v$ and thus $y = u$.
\end{proof}

The following standard example will be used often to define projectional factors of an operator $T:L_1(X)\to L_1(X)$.
\begin{exa}
\label{standard projectioning}
Let $(\widetilde h_I)_{I\in\mathcal{D}^+}$, $(\widehat h_L)_{L\in\mathcal{D}^+}$ be a faithful Haar systems and let $X$ be  a Haar  system space. Take
\[Z = \overline{\langle \widetilde h_I:I\in\mathcal{D}^+\rangle}\subset L_1\text{ and }W = \overline{\langle \widehat h_I:I\in\mathcal{D}^+\rangle}\subset X.\]
Then the map $P:L_1(X)\to L_1(X)$ given by
\[P u = \sum_{I\in\mathcal{D}^+}\sum_{L\in\mathcal{D}^+}\big\langle \widetilde h_I\otimes \nu_L\widehat h_L,u \big\rangle |I|^{-1}\widetilde h_I\otimes \mu_L\widehat h_L\]
(recall that  $\mu_I=\|\chi_I\|^{-1}_X$ and $\nu_L=\|\chi_L\|_{X^*}^{-1}$)
is a norm-one projection onto $Z(X) = \overline{\langle \widetilde h_I\otimes\widehat h_L: I,L\in\mathcal{D}^+ \rangle}$ and the map
\[A:L_1(X)\to L_1(X)\text{ given by }A(h_I\otimes h_L) = \widetilde h_I\otimes \widehat h_L\]
is a linear isometry onto $Z(X)$. Then, any bounded linear operator $T:L_1(X)\to L_1(X)$ is a 1-projectional factor of $S = A^{-1}PTA:L_1(X)\to L_1(X)$, so that for all $I,J,L,M\in\mathcal{D}^+$ we have
\[\big\langle h_I \otimes h_L,S\big(h_J\otimes h_M\big) \big\rangle = \big\langle \widetilde h_I \otimes \widehat h_L,T\big(\widetilde h_J\otimes\widehat h_M\big) \big\rangle.\]
\end{exa}

\begin{prop}
\label{gggame}
Let $\mathscr{A}\subset[\mathcal{D}^+]$ be a subset that has positive measure. Denote by $\mathcal{A} = \cup_{k_0=0}^\infty\{I_{k_0}:(I_k)_{k=0}^\infty\in\mathscr{A}\}$ and $Y_\mathscr{A} = \overline{\langle\{h_I:I\in\mathcal{A}\}\rangle}$. Then, there exits a subspace $Z$ of $Y_\mathscr{A}$ which is isometrically isomorphic to $L_1$ and $1$-complemented in $L_1$.
\end{prop}

\begin{proof} By approximating $\mathscr{A}$ in measure by closed sets from the inside, we can assume that $\mathscr{A}$ is closed.
For $k\in\N$ let $A_k = \cup\{I:I\in\mathcal{A}\cap\mathcal{D}_k\}$, and  $\mathscr{A_k} =\{ (I_n) \in [\cD^+]: I_k\in\cD_k, I_k\subset A_k\}$. 
Then it follows that $\mathscr{A}=\bigcap_k \mathscr{A}_k$ and
letting  $A = \cap_k A_k$, we deduce that 
$$|A|= \lim_{k\to\infty}|A_k|= \lim_{k\to\infty} |\mathscr{A}_k |=  |\mathscr{A}.$$ 
But also, for any $J\notin\mathcal{A}$, we have $J\cap A=\emptyset$. It follows that for any  $f\in L_1$ with $f|_{A^c} = 0$ and $J\notin\mathcal{A}$ we have $\langle h_J, f\rangle = 0$ and thus $f\in Y_\mathscr{A}$. In particular, the restriction operator $R_A:L_1\to L_1$ is a 1-projection onto a subspace that is isometrically isomorphic to $L_1$.
\end{proof}

The above proposition leads to the following example, which will be useful in the sequel.
\begin{exa}
\label{substandard projectioning}
Let $\mathscr{A}\subset[\mathcal{D}^+]$ be a subset that has positive measure and let $X$ be a Banach space. Then, there exists a subspace $Z$ of $Y_\mathscr{A}$ that is isometrically isomorphic to $L_1$ via $A:L_1\to Z$ and 1-complemented in $L_1$ via $P:L_1\to Z$. In particular, for any Banach space $X$ the space
\[Z(X) = \overline{Z\otimes X}\subset L_1(X)\]
is isometrically isomorphic to $L_1(X)$ via $A\otimes I:L_1(X)\to Z(X)$ and 1-complemented in $L_1(X)$ via $P\otimes I$.
\end{exa}

\subsection{Decompositions of operators on  $L_1(X)$}
We begin by listing further standard facts about projective tensor products. We then use these facts to associate to each bounded linear operator $T\!:\!L_1(X)\to L_1(X)$ a family of bounded linear operators on $L_1$. In the next section we will study compactness properties of this family. In later sections we use these properties to extract information about projectional factors of the operator $T$.

Let $E$, $F$ be Banach spaces.
\begin{enumerate}[leftmargin=19pt,label=(\alph*)]

\item\label{standard fare tensor b} For every $e_0^*\in E^*$ and $f_0^*\in F^*$ we may define the bounded linear maps
$q_{(e_0^*)}:E\otimes_\pi F\to F$ and $q^{(f_0^*)}:E\otimes_\pi F\to E$ given by $q_{(e_0^*)}(e\otimes f) = e_0^*(e)f$ and  $q^{(f_0^*)}(e\otimes f) = f_0^*(f)e$. Then, $\|q_{(e_0^*)}\| = \|e_0^*\|$ and $\|q^{(f_0^*)}\| = \|f_0^*\|$.

\item\label{standard fare tensor c} For every $e_0\in E$ and $f_0\in F$ we may define the maps
$j_{(e_0)}:F\to E\otimes_\pi F$ and $j^{(f_0)}:E\to E\otimes_\pi F$ given by $j_{(e_0)}f = e_0\otimes f$ and  $j^{(f_0)}e = e\otimes f_0$. Then, $\|j_{(e_0)}\| = \|e_0\|$ and $\|j^{(f_0)}\| = \|f_0\|$.

\item\label{standard fare tensor d} For every bounded linear operator $T:E\otimes_\pi F\to E\otimes_\pi F$, $f_0^*\in F^*$, and $f_0\in F$ the map $T^{(f_0^*,f_0)}:=q^{(f_0^*)}Tj^{(f_0)}:E\to E$ is the unique bounded linear map so that for all $e^*\in E^*$ and $e\in E$ we have $\langle e^*, T^{(f_0^*,f_0)}e\rangle = \langle e^*\otimes f_0^*, T(e\otimes f_0)\rangle$.

\item\label{standard fare tensor e} For every bounded linear operator $T:E\otimes_\pi F\to E\otimes_\pi F$, $e_0^*\in E^*$, and $e_0\in E$ the map $T_{(e_0^*,e_0)}:=q_{(e_0^*)}Tj_{(e_0)}:F\to F$ is the unique bounded linear map so that for all $f^*\in F^*$ and $f\in F$ we have $\langle f^*, T_{(e_0^*,e_0)}f\rangle = \langle e_0^*\otimes f^*, T(e_0\otimes f)\rangle$.

\end{enumerate}

\begin{ntt}
Let $X$ be  a Haar  system space. For $L\in\mathcal{D}^+$ we denote 
\begin{enumerate}[leftmargin=23pt,label=(\roman*)]

\item $q^L = q^{(\nu_Lh_L)}:L_1(X)\to L_1$,

\item $j^L = j^{(\mu_Lh_L)}:L_1\to L_1(X)$, and

\item $P^L = j^Lq^L:L_1(X)\to L_1(X)$.
\end{enumerate}
Note that, for any $k\in\N$, $\|\sum_{\{L:\iota(L)\leq k\}}P^L\| = 1$. This is because this operator coincides with $I\otimes P^{[\iota\leq k]}$, where $P^{[\iota\leq k]}:X\to X$ is the basis projection onto $(\mu_Lh_L)_{\iota(L)\leq k}$ (this is easy to verify on vectors of the form $u =h_I\otimes h_L$ whose linear span is dense in $L_1(X)$). We may therefore state the following.
\end{ntt}

\begin{rem}
Let $X$ be  a Haar  system space.
\begin{enumerate}[leftmargin=20pt,label=(\roman*)]

\item For each $L\in\mathcal{D}^+$, $P^L$ is a projection with image
\[Y^L = \{f\otimes (\mu_Lh_L): f\in L_1\}\]
that is isometrically isomorphic to $L_1$.  

 \item $(Y^L)_{L\in\mathcal{D}^+}$ forms a monotone Schauder decomposition of $L_1(X)$. In particular, for every $u\in L_1(X)$
 \[u = \sum_{L\in\mathcal{D}^+} P^Lu =\sum_{L\in\mathcal{D}^+} (q^Lu)\otimes (\mu_Lh_L).\]
Thus, $u$ admits a unique representation $u = \sum_{L\in\mathcal{D}^+}f_L\otimes(\mu_Lh_L)$.
\end{enumerate}
\end{rem}

\subsection{Operators on $L_1$ associated to an operator on $L_1(X)$} For  a Haar  system space $X$, we represent every bounded linear operator $T:L_1(X)\to L_1(X)$ as a matrix of operators $(T^{(L,M)})_{(L,M)\in\mathcal{D}^+}$, each of which is defined on $L_1$.
\begin{ntt}
Let $X$ be  a Haar  system space and let $T:L_1(X)\to L_1(X)$ be a bounded linear operator. For $L,M\in\mathcal{D}^+$ we denote $T^{(L,M)} = T^{(\nu_Lh_L,\mu_Mh_M)}$
(recall from Proposition \ref{RI properties} that  scalars $\mu_M$ and $\nu_L$ positive, and chosen  so that 
 $\mu_Mh_M$  is normalized in $X^*$ and $\nu_L h_L$ is normalized in $X$), so that for every $u\in L_1(X)$ we have
\begin{align}
\begin{split}
Tu &= \sum_{L\in\mathcal{D}^+}P^LT\Big(\sum_{M\in\mathcal{D}^+}P^Mu\Big) = \sum_{L\in\mathcal{D}^+}\sum_{M\in\mathcal{D}^+}j^LT^{(L,M)}q^Mu\\
&= \sum_{L\in\mathcal{D}^+}\sum_{M\in\mathcal{D}^+}\Big(T^{(L,M)}(q^Mu)\Big)\otimes(\mu_Lh_L).\label{matrix representation}
\end{split}
\end{align}
For $L\in\mathcal{D}^+$ we denote $T^L = T^{(L,L)}$.
\end{ntt}

The following type of operator is essential as it is easier to work with. A big part of the paper is to show that, within the constraints of the problem under consideration, every operator $T:L_1(X)\to L_1(X)$ is a 1-projectional factor with error $\e$ of an $X$-diagonal operator.
\begin{dfn}
\label{X-diagonal def}
Let $X$ be  a Haar  system space. A bounded linear operator $T:L_1(X)\to L_1(X)$ is called {\em $X$-diagonal} if for all $L\neq M\in\mathcal{D}^+$, $T^{(L,M)} = 0$. We then call $(T^L)_{L\in\mathcal{D}^+}$ the entries of $T$.
\end{dfn}

Note that $T$ is $X$-diagonal if and only if for all $f\in L_1$ and $L\in\mathcal{D}^+$ we have $T(f\otimes(\mu_Lh_L)) = (T^Lf)\otimes(\mu_Lh_L)$ if and only if for all $L\in\mathcal{D}^+$ the space $Y^L$ is $T$-invariant.

\begin{rem}
\label{in practice close to X-diagonal}
If $X$ is  a Haar  system space and $T:L_1(X)\to L_1(X)$ is a bounded linear operator so that $\sum_{L\neq M}\|T^{(L,M)}\| = \e <\infty$, then \eqref{matrix representation} yields that there exists an $X$-diagonal operator $\bar{T}:L_1(X)\to L_1(X)$ with entries $(T^L)_{L\in\mathcal{D}^+}$ so that $\|T - \bar{T}\| \leq \e$.
\end{rem}

\section{Compactness properties of families of operators}
In this section we extract compactness properties of families of operators associated to a $T:L_1(X)\to L_1(X)$. These results will be eventually applied to families that resemble ones of the form $(T^{(L,M)})_{(L,M)\in\mathcal{D}^+}$. The achieved compactness will later be used in a regularization process that will allow us to extract ``nicer'' operators that projectionally factor through $T$. We have chosen to present this section in a more abstract setting that permits more elegant statements and proofs.

\subsection{WOT-sequentially compact families} Taking WOT-limits of certain sequences of operators of the form $T^{(x^*,x)}$ is an important component of the proof. This element was already present in the approach of Capon \cite{CAPLL,CAPLX}. 

This essential Lemma due to Rosenthal is necessary in this subsection as well as the next one. A proof can be given, e.g., by induction on $j$ for $\e = 2^{-j}\sup_n\|\xi_n\|_1$.
\begin{lem}(\cite[Lemma 1.1]{rosenthal:1970})
\label{rosenthal lemma}
Let $(\xi_n)_n$ be a bounded sequence of elements of $\ell_1$ and $\e>0$. Then, there exits an infinite set $N = \{n_j:j\in\N\}\in[\N]^\infty$ so that for every $j_0\in\N$ we have $\sum_{j\neq j_0}|\xi_{n_{j_0}}(n_j)| \leq \e$.
\end{lem}

Here, WOT stands for the weak operator topology in $L_1(X)$.

\begin{thm}
\label{relatively compact orbit}
Let $X$ be a Banach space, $T:L_1(X)\to L_1(X)$ be a bounded linear operator, and $A$, $B$ be bounded subsets of $X^*$ and $X$, respectively. Assume that $B$ contains no sequence that is equivalent to the unit vector basis of $\ell_1$. Then, for every $f\in L_1$ the set
\[\big\{T^{(x^*,x)}f: (x^*,x)\in A\times B\big\}\]
is a uniformly integrable (and thus weakly relatively compact) subset of $L_1$. In particular, every sequence in $\{T^{(x^*,x)}: (x^*,x)\in A\times B\}$ has a WOT-convergent subsequence.
\end{thm}

\begin{proof}
The ``in particular'' part follows from the separability of $L_1$ and the fact that the set in question if bounded by $\|T\|\sup_{(x^*,x)\in A\times B}\|x^*\|\|x\|$.

Fix a sequence $(x_n^*,x_n)\in A\times B$. Assume that $(T^{(x_n^*,x_n)}f)_n$ is not uniformly integrable. Then, after passing to a subsequence, there exist $\de>0$ and a sequence of disjoint measurable subsets $(A_n)_n$ of $[0,1)$ so that for all $n\in\N$ we have
\[\de\leq \Big|\int_{A_n}\big(T^{(x_n^*,x_n)}f\big)(s)ds\Big| = \big|\langle \chi_{A_n},T^{(x_n^*,x_n)}f\rangle\big| = \big|\langle \chi_{A_n}\otimes x_n^*, T(f\otimes x_n)\rangle\big|.\]
For every $n\in\N$ define the scalar sequence $\xi_n = (\xi_n(m))_m$ given by $\xi_n(m) = \langle\chi_{A_m}\otimes x_m^*,T(f\otimes x_n)\rangle$. Then for every $m_0\in\N$ we have that for appropriate scalars $(\zeta_m)_{m=1}^N$ of modulus one
\begin{align}
\label{bounded in little ell1}
\sum_{m=1}^{m_0}|\xi_n(m)| &= \Big|\big\langle\sum_{m=1}^{m_0}\chi_{A_m}\otimes \zeta_mx_m^*,T(f\otimes x_n)\big\rangle\Big|\\\nonumber
& \leq \underbrace{\Big\|\sum_{m=1}^{m_0}\chi_{A_{m}}\otimes (\zeta_m x^*_{m})\Big\|}_{= \max_{1\leq m\leq m_0}\|x^*_{m}\|}\|T\|\|f\| \|x_n\|\nonumber\\
&\leq \|T\|\|f\|\sup_{(x^*,x)\in A\times B}\|x^*\|\|x\|.\nonumber
\end{align}
By Rosenthal's Lemma \ref{rosenthal lemma}, there exists an infinite subset $N = \{n_j:j\in\N\}$ of $\N$ so that for all $i_0\in\N$ we have $\sum_{j\neq i_0}|\xi_{n_{i_0}}(n_j)| \leq \de/2$. After relabelling, for all $n_0\in\N$ we have
\[\sum_{m\neq n_0}|\xi_{n_0}(m)|\leq \de/2.\]
We now show that $(x_n)_n$ is equivalent to the unit vector basis of $\ell_1$. Fix scalars $a_1,\ldots,a_N$. For appropriate scalars $\theta_1,\ldots,\theta_N$ of modulus 1 we have
\begin{equation}
\label{relatively compact orbit eq1}
\sum_{n=1}^Na_n\theta_n\langle \chi_{A_{n}}\otimes x^*_{n},T(f\otimes x_{n})\rangle \geq \de\sum_{n=1}^N|a_n|.
\end{equation}
Put
\begin{align*}
\Lambda & = \Big|\big\langle\sum_{m=1}^N\chi_{A_{m}}\otimes (\theta_m x^*_{m}),T\big(\sum_{n=1}^Nf\otimes a_nx_{n}\big)\big\rangle\Big|\\
&\leq \underbrace{\Big\|\sum_{m=1}^N\chi_{A_{m}}\otimes (\theta_m x^*_{m})\Big\|}_{= \max_{1\leq m\leq N}\|x^*_{m}\|}\|T\|\|f\| \Big\|\sum_{n=1}^Na_nx_n\Big\|\\
&\leq \|T\|\|f\|\sup_{x^*\in A}\|x^*\|\Big\|\sum_{n=1}^Na_nx_n\Big\| . 
\end{align*}
Also,
\begin{align*}
\La & = \Big|\sum_{n=1}^Na_n\theta_n\langle \chi_{A_{n}}\otimes x^*_{n},T(f\otimes x_{n})\rangle + \sum_{n=1}^Na_n\sum_{m\neq n}\theta_m\langle \chi_{A_{m}}\otimes x^*_{m},T(f\otimes x_{n})\rangle\Big|\\
&\geq \de\sum_{n=1}^N|a_n| - \sum_{n=1}^N|a_n|\sum_{m\neq n}|\xi_n(m)| \geq \de/2\sum_{n=1}^N|a_n|.
\end{align*}
Thus, $\|\sum_{n=1}^Na_nx_n\| \geq c\sum_{n=1}^N|a_n|$, where $c = \de/(2\|T\|\|f\|\sup_{x^*\in A}\|x^*\|)$.
\end{proof}

\subsection{Compactness in operator norm} We discuss families that are uniformly eventually close to multipliers and how to obtain compact sets from them. This is particularly important in the sequel because compactness will be essential in achieving strong stabilization properties of operators $T:L_1(X)\to L_1(X)$.

\begin{ntt}
For $n\in\N$ we denote by $P_{(\leq n)}:L_1\to L_1$ the norm-one canonical basis projection onto $\langle\{h_I:I\in\mathcal{D}^n\}\rangle$. We also denote $P_{(>n)} = I - P_{(\leq n)}$.
\end{ntt}

\begin{dfn}
\label{uechm def}
A set $\mathscr{T}$ of bounded linear operators on $L_1$ is called {\em uniformly eventually close to Haar multipliers} if there exists a collection $(D_T)_{T\in\mathscr{T}}$ in $\mathcal{L}_{HM}(L_1)$ so that
\begin{equation*}
\lim_n\sup_{T\in\mathscr{T}}\Big(\|(T-D_T)P_{(>n)}\| + \|P_{(>n)}(T-D_T)\|\Big) = 0.
\end{equation*}
\end{dfn}

The main result of this subsection is the first one in the paper that requires a certain amount of legwork.
\begin{thm}[Fundamental Lemma]
\label{fundamental lemma}
Let $X$ be a Banach space, $A$, $B$ be bounded subsets of $X^*$ and $X$, respectively, and $C\subset A\times B$. Let $T:L_1(X)\to L_1(X)$ be a bounded linear operator and assume the following.
\begin{enumerate}[leftmargin=20pt,label=(\roman*)]

\item The set $B$ contains no sequence that is equivalent to the unit vector basis of $\ell_1$.

\item The set $\{T^{(x^*,x)}:(x^*,x)\in C\}$ is uniformly eventually close to Haar multipliers.

\end{enumerate}
Then, for every $\eta>0$, there exits a closed subset $\mathscr{A}$ of $[\mathcal{D}^+]$ with $|\mathscr{A}|>1-\eta$ so that the set $\{T^{(x^*,x)}P_{\mathscr{A}}:(x^*,x)\in C\}$ is relatively compact in the operator norm topology.
\end{thm}

\begin{rem}
  It is not hard to see that the unit ball of $\mathcal{L}_{HM}(L_1)$ is a compact set in the strong operator topology of $L_1$. In fact, this is the $w^*$-topology inherited by a predual of $\mathcal{L}_{HM}(L_1)$, namely Rosenthal's Stopping Time space studied by Bang and Odell in \cite{bang:odell:1985,bang:odell:1989}, by Dew in \cite{dew:2003}, and by Apatsidis in \cite{apatsidis:2015}.
 The Fundamental Lemma (Theorem \ref{fundamental lemma}) states is that under the right conditions, strong operator convergence yields convergence in operator norm on a big subspace of $L_1$. Therefore, this is a type of Egorov Theorem. We point out that some restriction to the family of operators is necessary for the conclusion to hold. If one takes for example $D_n = P_{(\leq n)}$ then this converges to $I$ in the strong operator topology. Yet, for no non-empty set of branches $\mathscr{A}$  the set $\{D_nP_\mathscr{A}:n\in\N\}$ is relatively compact in the operator norm topology.
\end{rem}

\begin{lem}
\label{quantifiably non-uniformly integrable}
Let  $r >0$, $(I_k)_{k=0}^\infty\in[\mathcal{D}^+]$ associated to $(\theta_k)_{k=1}^\infty\in\{-1,1\}^\N$, and $(a_k^n)_{(k,n)\in(\{0\}\cup\N)\times\N}$ be a collection of scalars. Assume that there exist $k_1<\ell_1<k_2<\ell_2<\cdots$ so that for each $n\in\N$ we have
\[\sum_{k = k_n+1}^{\ell_n}|a^n_{k} - a^n_{k-1}| \geq r  . \]
For every $\ell,n\in\N$ define $f_n^\ell = \sum_{k=0}^\ell a_k^n\theta_k|I_k|^{-1}h_{I_k}$. Then, there exists  a strictly increasing  sequence of disjoint measurable subsets $(A_n)_n$ of $[0,1)$ so that for all $n\in\N$ and $\ell\geq \ell_n$ we have
\begin{equation*}
f_{n}^{\ell}(s) = f_{n}^{\ell_n}(s)\text{ on }A_n\text{ and }\int_{A_n}\big|f_{n}^{\ell_n}(s)\big|ds \geq r/3.
\end{equation*}
\end{lem}

\begin{proof} Let $(B_k)$ be the partition of $[0,1)$, defined by $B_k=I_k\setminus I_{k+1}$, $k\in \N$.
We conclude from  the inequality \eqref{branches variation 1a} in Proposition \ref{branches variation} that:
\begin{enumerate}[leftmargin=20pt,label=(\roman*)]

\item\label{quantifiably non-uniformly integrable 1} for every $k \leq \ell_n \leq \ell\in\N$ and $s\in B_k$ we have $f_n^\ell(s) = f^{\ell_n}_n(s)$ and

\item\label{quantifiably non-uniformly integrable 2} for every $m\leq \ell_n\in\N$ we have
\[\int_{\cup_{k=m}^{\ell_n} B_k}|f_n^{\ell_n}(s)|ds \geq \frac{1}{3}\sum_{k=m+1}^{\ell_n} |a_k - a_{k-1}|.\]

\end{enumerate}
Put $A_n = \cup_{i=k_n}^{\ell_n}B_i$. The conclusion follows directly from \ref{quantifiably non-uniformly integrable 1} and \ref{quantifiably non-uniformly integrable 2}.
\end{proof}

\begin{proof}[Proof of Theorem \ref{fundamental lemma}]
Put $\mathscr{T} = \{T^{(x^*,x)}:(x^*,x)\in A\times B\}$. Take a family $(D_T)_{T}$ that witnesses Definition \ref{uechm def}. For each $T\in\mathscr{T}$ we have
\begin{align}
\|(T - D_T)P_{(>k)}\| &\leq \sup_{S\in\mathscr{T}}\Big(\big\|(S - D_S)P_{(>k)}\big\|\Big) =\e_k .\\
\|P_{(>k)}TP_{(\leq k)}\| &\leq \underbrace{\big\|P_{(>k)}D_TP_{(\leq k)}\big\|}_{=0} +  \big\|P_{(>k)}(T - D_T)P_{(\leq k)}\big\|\nonumber\\
&\leq \big\|P_{(>k)}(T - D_T)\|\\ &\leq \sup_{S\in\mathscr{T}}\Big(\big\|P_{(>k)}(S- D_S)\big\|\Big) \!=\!\de_k.\nonumber
\end{align}
Both $(\e_k)_k$ and $(\de_k)_k$ tend to zero. For each $T\in\mathscr{T}$ denote by $(a^T_I)_{I\in\mathcal{D}^+}$ the entries of $D_T$.

{\em Claim:} Fix $\sigma = (I_k)_{k=1}^\infty\in[\mathcal{D}^+]$ and $r>0$. Then, there exists $k_0\in\N$ so that for all $T\in\mathscr{T}$ we have $\sum_{k=k_0}^\infty|a^T_{I_k} - a^T_{I_{k-1}}| \leq r$.

We will assume that the claim is true and proceed with the rest of the proof. For every $N,k_0,\in\N$ let
\[\mathcal{A}_{N,k_0} = \Big\{\sigma = (I_k)_{k=1}^\infty\in[\mathcal{D}^+]: \sup_{T\in\mathscr{T}}\sum_{k=k_0}^\infty|a^T_{I_k} - a^T_{I_{k-1}}| \leq 2^{-N}\Big\},\]
which is a closed subset of $[\mathcal{D}^+]$ and by the claim we have $\cup_{k_0}\mathscr{A}_{N,k_0} = [\mathcal{D}^+]$, for all $N\!\in\!\N$. We may therefore pick a strictly increasing sequence of natural numbers $(k_N)$ so that for each $N$ we have $|\mathscr{A}_{N,k_N}| \geq 1-\eta/2^N$. We put $\mathscr{A} = \cap_N\mathscr{A}_{N,k_N}$ and we demonstrate that this is the desired set.

To show that $\{TP_\mathscr{A}:T\in\mathscr{T}\}$ is relatively compact with respect to the operator norm we fix  $\e>0$ and $(T_n)_n$ in $\mathscr{T}$. For each $n\in\N$ denote $D_n = D_{T_n}$. We will find  $M\in[\N]^{\infty}$ so that for all $n,m\in M$ we have $\|T_nP_\mathscr{A} - T_mP_\mathscr{A}\| \leq 11\e$. Fix $N\in\N$ so that $2^{-N}\leq \e$, $\e_{k_N}\leq \e$, and $\de_{k_N}\leq\e$. For each $n\in\N$ write
\begin{equation*}
T_n = D_nP_{(>k_N)} + \underbrace{(T_n-D_n)P_{(>k_N)}}_{=:A_n}+\underbrace{P_{(>k_N)}T_nP_{(\leq k_N)}}_{=:B_n} + \underbrace{P_{(\leq k_N)}T_nP_{(\leq k_N)}}_{=:C_n}.
\end{equation*}
Then we have $\|A_n\| \leq \e_{k_N}\leq \e$ and $\|B_n\| \leq \de_{k_N}\leq \e$. By passing to
 a subsequence of $(T_n)$
we may assume that for all $n,m\in\N$ we have
 (letting  $a_I^n=a^{T_n}_I$)
\begin{equation}
\sum_{\left\{\substack{I\in\mathcal{D}^+\\|I|\geq 1/2^{k_N+1}}\right\}}|a_I^n - a_I^m| \leq \e.
\end{equation}
Since the $C_n$ are bounded elements of a finite dimensional space, we can also assume  that
 $\|C_n - C_m\| \leq \e$, for $m,n\in \N$.  Therefore, for $n,m\in \N$ we have
\begin{align*}
\|T_nP_\mathscr{A} - T_mP_\mathscr{A}\| \leq \underbrace{\| D_nP_{(>k_N)}P_\mathscr{A} - D_mP_{(>k_N)}P_\mathscr{A} \|}_{=:\Lambda} + 5\e.
\end{align*}
Luckily, the remaining quantity $\Lambda$ is the norm of a Haar multiplier on $L_1$ and we know how to compute this. If for $\sigma = (I_k)_{k=0}^\infty\in\mathscr{A}$ we put
\begin{alignat*}{2}
\Lambda_\sigma &= \sum_{k=k_N+1}^\infty|(a^n_{I_k}-a^m_{I_k}) - (a^n_{I_{k-1}}-a^m_{I_{k-1}})| + |a^n_{I_{k_N}} - a^m_{I_{k_N}}| + \lim_k|a^n_{I_k}-a^m_{I_k}|\\
&\leq 2\underbrace{\sum_{k=k_N+1}^\infty|(a^n_{I_k}-a^m_{I_k}) - (a^n_{I_{k-1}}-a^m_{I_{k-1}})|}_{\leq 2/2^N\leq 2\e} + 2\underbrace{|a^n_{I_{k_N}} - a^m_{I_{k_N}}|}_{\leq \e}
\leq6\e.
\end{alignat*}
Then, by Remark \ref{restricted triple bar}, $\Lambda = \sup_{\sigma\in\mathscr{A}}\Lambda_\sigma$ and thus $\|T_nP_\mathscr{A} - T_mP_\mathscr{A}\|\leq 11\e$.

We now provide the owed proof of the claim. We fix $\sigma = (I_k)_{k=0}^\infty$, with associated signs $(\theta_k)_{k=0}^\infty$. Let us assume that the conclusion fails. Then, we may find $(T_n)_n = (T^{(x^*_n,x_n)})_n$ in $\mathscr{T}$, each $T_n$ is associated with a $D_n$ (each $D_n$ has entries $(a_I^n)_{I\in\mathcal{D}^+}$), and $k_1 < \ell_1<k_2<\ell_2<\cdots$ so that for all $n\in\N$
\begin{equation*}
\sum_{k=k_n+1}^{\ell_n}|a^{n}_{I_k} - a^{n}_{I_{k-1}}|\geq r.
\end{equation*}
Pick $k_0\in\N$ so that $\e_{k_0}\leq r/12$. For $k,n\in\N$ define $b_k^n = 0$ if $k\leq k_0$ and $b_k^n = a_{I_k}^{n}$ if $k>k_0$. If we additionally assume that $k_1>k_0$ then for all $n\in\N$ we have
\begin{equation}
\label{ucmycse3}
\sum_{k=k_n+1}^{\ell_n}|b_k^n - b^n_{k-1}|\geq r.
\end{equation}
For each $n,\ell\in\N$ put
\[f_n^\ell = \sum_{k=0}^\ell b_k^n\theta_k|I_k|^{-1}h_{I_k} = D_nP_{(>k_0)}(\underbrace{|I_{\ell+1}|^{-1}\chi_{I_{\ell+1}}}_{=:\psi_\ell}).\]
By Lemma \ref{quantifiably non-uniformly integrable} we may find a sequence of $(A_n)_n$ of disjoint measurable sets so that for each $n\in\N$ the sequence $(f_n^\ell(s))_{\ell\geq \ell_n}$ is constant for all $s\in A_n$ and $\|f_n^{\ell_n}|_{A_n}\|_{L_1} \geq r/3$. For each $n\in\N$ fix $g_n$ in the unit sphere of $L_\infty$ with support in $A_n$ so that for all $\ell\geq \ell_n$
\begin{align*}
r/3&\leq \|f_n^{\ell_n}|_{A_n}\|_{L_1} = \big|\langle g_n ,f_n^{\ell}\rangle\big| = \big|\langle g_n, D_nP_{(>k_0)}(\psi_\ell)\rangle\big|\\
\nonumber &\leq   \big|\langle g_n, T_nP_{(>k_0)}(\psi_\ell)\rangle\big| + r/12  = \big|\langle g_n\otimes x_n^*, T\big(\underbrace{(P_{(>k_0)}\psi_\ell)}_{\phi_\ell}\otimes x_n\big)\rangle\big| + r/12.
\end{align*}
Note that for all $\ell\in\N$, $\|\phi_\ell\|_{L_1}\leq 2$. Then, for all $n\in\N$ and $\ell\geq\ell_n$
\begin{equation*}
\big|\langle g_n\otimes x_n^*, T\big(\phi_\ell\otimes x_n\big)\rangle\big| \geq r/4.
\end{equation*}
Pick an  $L\in[\N]^\infty$ so that for each $m,n\in\N$ the limit
\[
\xi_n(m) := \lim_{\ell\in L}\langle g_m\otimes x_m^*, T(\psi_\ell\otimes x_n) \rangle \text{ exists}.
\]
Because the sequence $(g_m)_m$ is disjointly supported, an identical calculation as in \eqref{bounded in little ell1} yields that for all $n\in\N$ we have $$\sum_m|\xi_n(m)| \leq 2\|T\|\sup_{(x^*,x)\in C}\|x^*\|\|x\|.$$ Thus, by Rosenthal's Lemma \ref{rosenthal lemma} we may pass to a subsequence and relabel so that for all $n_0\in\N$ we have $\sum_{m\neq n_0}|\xi_{n_0}(m)| \leq r/8$.

We will show that $(x_n)_n$ must be equivalent to the unit vector basis of $\ell_1$, which would contradict our assumption and thus finish the proof.  Fix scalars $a_1,\ldots,a_N$ and for $\ell\in L$ with $\ell\geq \ell_N$ pick appropriate scalars $\zeta^\ell_1,\ldots,\zeta^\ell_N$ of modulus one so that  we have
\[
\frac{r}{4}\sum_{n=1}^N|a_n| \leq \sum_{n=1}^N\langle g_n\otimes\big(\zeta^\ell_nx_n^*\big), T\big(\phi_\ell\otimes(a_nx_n)\big)\rangle
\]
and put
\begin{align*}
\La_\ell &= \big|\big\langle\sum_{m=1}^N g_n\otimes\big(\zeta^\ell_mx_m^*\big),\sum_{n=1}^NT\big(\phi_\ell\otimes(\sum_{n=1}^Na_nx_n)\big)\big\rangle\big|\\
&\leq \Big(2\|T\|\sup_{x^*\in A}\|x^*\|\Big)\big\|\sum_{n=1}^N a_nx_n\big\|.
\end{align*}
But also,
\begin{align*}
\lim_{\ell\in L}\La_\ell &= \lim_{\ell\in L}\Big| \sum_{n=1}^N\langle g_n\otimes\big(\zeta^\ell_nx_n^*\big), T\big(\phi_\ell\otimes(a_nx_n)\big)\rangle\\
&+ \sum_{n=1}^Na_n\sum_{m\neq n}\zeta_m^\ell \langle g_m\otimes x_m^*, T\big(\psi_\ell\otimes x_n\big) \rangle\Big|\\
&\geq \frac{r}{4}\sum_{n=1}^N|a_n| - \sum_{n=1}^N|a_n|\sum_{m\neq n}|\xi_n(m)| \geq \frac{r}{8}\sum_{n=1}^N|a_n|.
\end{align*}
Therefore, $\|\sum_{n=1}^Na_nx_n\| \geq r/(16\|T\|\sup_{x^*\in A}\|x^*\|)\sum_{n=1}^N|a_n|$.
\end{proof}

\section{Projectional factors of $X$-diagonal operators}

The main purpose of the section is to prove the following first step towards the final result. The Fundamental Lemma (Theorem \ref{fundamental lemma}) is a necessary part of the proof.

\begin{thm}
\label{arbitrary to X-diagonal}
Let $X$ be in $\cH^*$ and let $T:L_1(X)\to L_1(X)$ be a bounded linear operator. Then, for every $\e>0$, $T$ is a $1$-projectional factor with error $\e$ of an $X$-diagonal operator $S:L_1(X)\to L_1(X)$.
\end{thm}

The strategy is to first pass to an operator $S$ with the family $(S^{(L,M)})_{L\neq M}$ uniformly eventually close to Haar multipliers (in reality, $S$ satisfies something slightly stronger). We will then use the Fundamental Lemma to eliminate these entries altogether. The following result states how uniform eventual proximity to Haar multipliers is achieved in practice.

\begin{lem}
\label{uechm in practice}
Let $\mathscr{T}$ be a subset of $\mathcal{L}(L_1)$ and $(\e_{(I,J)})_{(I,J)\in\mathcal{D}^+\times\mathcal{D}^+}$ be a summable collection of positive real numbers. If for every $I\neq J\in\mathcal{D}^+$ and $T\in\mathscr{T}$ we have
$\big|\langle h_I, T\big(|J|^{-1}h_J\big)\rangle\big|\leq \e_{(I,J)}$
then $\mathscr{T}$ is uniformly eventually close to Haar multipliers.
\end{lem}

\begin{proof}
For fixed $T\in\mathscr{T}$ put $a_I = \langle h_I, T(|I^{-1}|h_I)\rangle$. This collection defines a bounded Haar multiplier $D_T$ because for all $f$ in the unit ball of $L_1$, $\|(T-D_T)f\| \leq \sum_{I\in\mathcal{D}^+}\sum_{\{J\in\mathcal{D}^+: J\neq I\}}|\langle h_I,T(|J|^{-1}h_J)\rangle| <\infty$. Also, for all $n\in\N$,
\begin{align*}
\Big\|TP_{(>n)} - D_TP_{(>n)}\Big\|& \leq \sum_{I\in\mathcal{D}^+}\sum_{J\in\mathcal{D}^+\setminus\mathcal{D}^n}\e_{(I,J)} =:\e_n\text{ and}\\
\Big\|P_{(>n)} T - P_{(>n)}D_T\Big\| & \leq \sum_{I\in\mathcal{D}^+\setminus\mathcal{D}^n}\sum_{J\in\mathcal{D}^+}\e_{(I,J)} =:\de_n.
\end{align*}
Both $(\e_n)_n$ and $(\de_n)_n$ tend to zero.
\end{proof}

The next lemma is the basic tool used to achieve the first step.

\begin{lem}
\label{far enough acts small}
Let $X$ be in $\cH^*$ and $\mathscr{T}\subset \mathcal{L}(X)$, $G\subset X^*$, and $F\subset X$ be finite sets. Then, for any $\e>0$, there exists $i_0\in\N$ so that for any disjoint collection $\De$ of $\mathcal{D}^+$ with $\min\iota(\De)\geq i_0$ and any $\theta\in\{-1,1\}^\De$ we have
\[
\max_{g\in G,T\in\mathscr{T}}\big|\big\langle g,T(h_\De^\theta)\big\rangle\big| \leq \e\text{ and }\max_{f\in F,T\in\mathscr{T}}\big|\big\langle h_\De^\theta,T(f)\big\rangle\big| \leq \e
\]
(recall that $h^{\theta}_\Delta$ was  introduced before Definition \ref{faithful Haar system}).
\end{lem}

\begin{proof}
The result is an immediate consequence of the following fact: let $(\De_k)$ be a sequence of finite disjoint collections of $\mathcal{D}^+$ with $\lim_k\min\iota(\De_k) = \infty$ and for every $k\in\N$ let $\theta_k\in\{-1,1\}^{\De_k}$.
\begin{enumerate}[leftmargin=19pt,label=(\alph*)]

\item The sequence $(h_{\De_k}^{\theta_k})_k$ is weakly null.

\item\label{far enough acts small ii} The sequence $(h_{\De_k}^{\theta_k})_k$ is a bounded block sequence in $X^*$ and thus it is $w^*$-null.

\end{enumerate}
There is nothing further to say about statement (b). We now explain how statement (a) is achieved. Note that any sequence of independent $\{-1,1\}$-valued random variables of mean $0$
is
distributionally equivalent to $(r_n)_n$ and thus weakly null. Any sequence as in statement (a) has a subsequence
  which is of the form $(\frac{r_n+r'_n}2)$, where $(r_n)$ and $(r'_n)$ are both sequences of independent $\{-1,1\}$-valued random variables of mean $0$.
  Thus, it is weakly null as well.
\end{proof}

We carry out the first step towards the proof of Theorem \ref{arbitrary to X-diagonal}

\begin{prop}
\label{pass to mead}
Let $X$ be in $\cH^{*}$ and denote by $C$ the set of all pairs $(g,f)$ in $B_{X^*}\times B_X$ so that $g$ and $f$ have  finite  and disjoint supports with respect to the Haar system. Then, every bounded linear operator $T:L_1(X)\to L_1(X)$ is a 1-projectional factor of a bounded linear operator $S:L_1(X) \to L_1(X)$ so that the family $\{S^{(f,g)}:(f,g)\in C\}$ is uniformly eventually close to Haar multipliers.
\end{prop}

\begin{proof}
We will  inductively construct faithful Haar systems $(\widetilde h_I)_{I\in\mathcal{D}^+}$ and $(\widehat h_L)_{L\in\mathcal{D}^+}$. In each step $k$ of the induction we will define $\widetilde h_I$ and then $\widehat h_L$ with $k = \iota(I) = \iota(L)$ (i.e., $I=L$ but we separate the notation for clarity). These vectors are of the form $\widetilde h_I = \sum_{J\in\De_I}h_J$ and $\widehat h_L = \sum_{M\in\Ga_L}h_M$. The inductive assumption is the following.

For every $J,J',M,M'\in\mathcal{D}^+$ with $\iota(J') \neq \iota(J)\leq k$ and $\iota(M')\neq \iota(M)\leq k$ we have
\begin{gather}
\label{arbitrary to mead I}
\big|\big\langle \widetilde h_J\otimes \nu_M\widehat h_M, T\big(|J'|^{-1}\widetilde h_{J'}\otimes \mu_{M'} \widehat h_{M'}\big)\big\rangle\big| \leq 2^{-(\iota(J)+\iota(J')+\iota(M)+\iota(M'))}.
\end{gather}

We may start by picking $\widetilde h_\emptyset = \widehat h_\emptyset = h_\emptyset$. We now carry out the $k$'th inductive step. Let $I = L\in\mathcal{D}^+$ with $\iota(I) = \iota(L) = k$. We will apply Lemma \ref{far enough acts small} twice, once for $L_1$ and once for $X$. First, define the following finite sets.
\begin{gather*}
\mathscr{T}_1 = \{T^{(\nu_{M}\widehat h_{M},\mu_{M'}\widehat h_{M'})}:\iota(M),\iota(M')<k\}\subset\mathcal{L}(L_1)\\
G_1 = \{\widetilde h_J: \iota(J)<k\}\subset (L_1)^* \text{ and }F_1 = \{|J|\widetilde h_J: \iota(J)<k\}\subset L_1.
\end{gather*}
Use Lemma \ref{far enough acts small} to pick $\widetilde h_I$ so that
\[
\max_{g\in G_1,T\in\mathscr{T}_1}\big|\big\langle g,T(\widetilde h_I)\big\rangle\big| \leq 2^{-4k}\text{ and }\max_{f\in F_1,T\in\mathscr{T}_1}\big|\big\langle \widetilde h_I,T(f)\big\rangle\big| \leq |I|2^{-4k}.
\]
Next, we take the finite sets
\begin{gather*}
\mathscr{T}_2 = \{T_{(\widetilde h_{J},|J'|^{-1}\widetilde h_{J'})}:\iota(J),\iota(J')\leq k\}\subset\mathcal{L}(X)\\
G_2 = \{ \nu_M\widehat h_M: \iota(M)<k\}\subset (L_1)^* \text{ and }F_2 = \{\mu_M\widehat h_M: \iota(M)<k\}\subset X.
\end{gather*}
Use Lemma \ref{far enough acts small} to pick $\widehat h_M$ so that
\[
\max_{g\in G_2,T\in\mathscr{T}_2}\big|\big\langle g,T(\widehat h_L)\big\rangle\big| \leq \mu_L^{-1}2^{-4k}\text{ and }\max_{f\in F_2,T\in\mathscr{T}_2}\big|\big\langle \widehat h_L,T(f)\big\rangle\big| \leq |I|\nu_L^{-1}2^{-4k}.
\]
The inductive step is complete and it is straightforward to check that the inductive hypothesis is preserved.

Take the operator $S$ given in Example \ref{standard projectioning}. We will show that it has the desired property. Fix $g\in B_{X^*}$, $f\in B_X$ with $g = \sum_{M\in E}b_M\nu_M h_M$ and $f = \sum_{M\in F}a_M\mu_Mh_M$ so that $F,G$ are finite and disjoint. Then,
\[S^{(g,f)} = \sum_{M\in E}\sum_{M'\in F}b_Ma_{M'}S^{(M,M')}\]
and for $I\neq J\in\mathcal{D}^+$ we have
\begin{align*}
\big|\big\langle h_I , S^{(g,f)}(|J|^{-1}h_J) \big\rangle\big|\span\leq \sum_{M\in E}\sum_{M'\in F}\big|\big\langle h_I , S^{(M,M')}(|J|^{-1}h_J) \big\rangle\big|\\
&= \sum_{M\in E}\sum_{M'\in F}\big|\big\langle h_I\otimes(\nu_Mh_M) , S\big((|J|^{-1}h_J)\otimes(\mu_{M'}h_{M'})\big) \big\rangle\big|\\
&= \sum_{M\in E}\sum_{M'\in F}\big|\big\langle\widetilde h_I\otimes(\nu_M\widehat h_M) , T\big((|J|^{-1}\widetilde h_J)\otimes(\mu_{M'}\widehat h_{M'})\big) \big\rangle\big|\\
&\leq \sum_{M\in E}\sum_{M'\in F}2^{-(\iota(J)+\iota(J')+\iota(M)+\iota(M'))}\leq 2^{-(\iota(I)+\iota(J))}=:\e_{(I,J)}.
\end{align*}
By Lemma \ref{uechm in practice}, the family under consideration is uniformly eventually close to Haar multipliers.
\end{proof}

\begin{rem}
Proposition \ref{pass to mead} can be achieved if we merely assume that $X$ is  a Haar  system space as condition \ref{nice RI class 1} of Definition \ref{nice RI class} can be replaced with a probabilistic argument.
We presented the slightly simpler proof that assumes \ref{nice RI class 1}. 
\end{rem}

We now eliminate the off-diagonal entries to obtain an $X$-diagonal operator that projectionally factors through $T$.
\begin{proof}[Proof of Theorem \ref{arbitrary to X-diagonal}]
By Proposition \ref{pass to mead}, $T$ is a 1-projectional factor of an $S:L_1(X)\to L_1(X)$ that satisfies  condition (ii) of Theorem \ref{fundamental lemma}.

For a  finite pairwise disjoint collection $\Gamma\subset \cD^+$  we define $\Gamma(n)=\{ D\in\cD_n: D\subset\Gamma^* \}$. Note that  $\Gamma(n) $ is a partition of
 $\Gamma^*$ for large enough $n$. Also note that from our condition $(*)$ it follows that  for  two finite subsets $\Gamma,\Gamma'$ of $\cD$, with $\Gamma^*\cap (\Gamma')^*=\emptyset$, the set 
 $$C^{(\Gamma, \Gamma')} =\{ (\nu_{(\Ga')^*} h_{\Ga'(n)}, \mu_{\Ga^*} h_{\Ga}): n\in\N\} \cup \{(\nu_{\Ga^*} h_{\Ga}, \mu_{(\Ga')^*} h_{\Ga'(n)}):n\in\N\big\}$$
 satisfies condition (i) of Theorem \ref{fundamental lemma}.
 The following claim will be the main step towards recursively defining an appropriate  faithful Haar system  $(\tilde h_L)$.
 
 \noindent{\em Claim.} There are  $\mathscr{A}\subset [\cD]$, with $|\mathscr{A}|>1 -\eta$, and  $\mathscr{U}\in[\N]^\infty$, so that
 \begin{align*}
\lim_{n\in\mathscr{U}}S^{(\nu_{(\Ga')^*} h_{\Ga'(n)}, \mu_{\Ga^*} h_{\Ga})}P_\mathscr{A} &=0\text{ and }
\lim_{n\in\mathscr{U}}S^{(\nu_{\Ga^*} h_{\Ga}, \mu_{(\Ga')^*} h_{\Ga'(n)})}P_\mathscr{A} &=0
\end{align*}
 with respect to the operator norm, for all $\Gamma,\Gamma'\subset \cD$, with $\Gamma^*\cap (\Gamma')^*=\emptyset$.
 
 In order to show the claim we choose for each  pair $(\Ga,\Ga')$,  with $\Gamma,\Gamma'\subset \cD$ being finite, $\eta_{(\Gamma,\Gamma')}>0$.
    with $\sum \eta_{(\Gamma,\Gamma')}<\eta$. Then, using Theorem \ref{fundamental lemma}, we choose
  a closed set $\mathscr{A}_{(\Ga,\Ga')}$ in $[\mathcal{D}^+]$ with $|\mathscr{A}_{(\Ga,\Ga')}| > 1-\eta_{(\Ga,\Ga')}$, so that $\{S^{(g,f)}P_{\mathscr{A}_{(\Ga,\Ga')}}:(g,f)\in C^{(\Ga,\Ga')}\}$ is relatively compact in the operator norm topology. Put $\mathscr{A} = \cap\mathscr{A}_{(\Ga,\Ga')}$ and note that $|\mathscr{A}|>1-\eta$ and that for each $(\Ga,\Ga')$ we still have that $\{S^{(g,f)}P_{\mathscr{A}}:(g,f)\in C^{(\Ga,\Ga')}\}$ is relatively compact.

Via a Cantor diagonalization find $\mathscr{U}\in[\N]^\infty$ so that for every pair $(\Ga,\Ga')$ both limits
\begin{align*}
 S^{(\Ga,\Ga')}_1\!\!:=\lim_{n\in\mathscr{U}}S^{(\nu_{(\Ga')^*} h_{\Ga'(n)}, \mu_{\Ga^*} h_{\Ga})}P_\mathscr{A}\text{ and }
S^{(\Ga,\Ga')}_2\!\!:=\lim_{n\in\mathscr{U}}S^{(\nu_{\Ga^*} h_{\Ga}, \mu_{(\Ga')^*} h_{\Ga'(n)})}P_\mathscr{A}
\end{align*}
exist with respect to the operator norm. As we will now see right away, $S^{(\Ga,\Ga')}_1 = S^{(\Ga,\Ga')}_2 = 0$. Indeed, for any $g\in L_\infty$ and $f\in L_1$ we have
\begin{align*}
\langle g, S^{(\Ga,\Ga')}_2f\rangle &= \lim_{n\in\mathscr{U}} \big\langle g, S^{(\nu_{\Ga^*} h_{\Ga}, \mu_{(\Ga')^*} h_{\Ga'(n)})}\big(P_\mathscr{A}f)\big)\big\rangle\\
 &= \lim_{n\in\mathscr{U}} \big\langle g\otimes (\nu_{\Ga^*}h_\Ga), S\big((P_\mathscr{A}f)\otimes (\mu_{(\Ga')^*}h_{\Ga'(n)})\big)\big\rangle = 0
\end{align*}
because $(h_{\Ga'(n)})_n$ is weakly null in $X$, by   Lemma \ref{far enough acts small}. With the same computation, $S^{(\Ga,\Ga')}_2 = 0$ because $(h_{\Ga'(n)})_n$ is $w^*$-null in $X^*$. This finishes the proof of the claim.

We now choose  inductively a faithful Haar system $(\widetilde h_L)_{L\in\mathcal{D}^+}$ so that for every $L\neq M\in\mathcal{D}^+$ we have 
\begin{equation}\label{what we want}\|S^{(\nu_L\widetilde h_L,\mu_M\widetilde h_M)}P_\mathscr{A}\| \leq \e 2^{-(\iota(L)+\iota(M))}.\end{equation} 
Assume $M\in\cD$ and  $\widetilde h_L= h_{\Gamma_L}$, has been chosen for all $L\in \cD^+$ with $\iota(L)<\iota(M)$, ($\widetilde h_\emptyset= h_\emptyset$ and 
$\widetilde h_{(0,1)}= h_{[0,1)}$ by definition). Without loss of generality we can assume that $M=K^+$ for some $K\in\cD$ with $\iota( K)<\iota(M)
 $. Thus we will choose $\Gamma_M$ so that 
 $\Gamma^*_M= [\widetilde h_K=1]$. For large enough $n_0\in\N$  it follows that  $[\widetilde h_K=1]  =(\Gamma' )^* $ for some $\Gamma'\subset \cD_{n_0}$.
 Then  we can use our claim that for large enough $n>0$, we  let $\Gamma_M= \Gamma'(n)$
  we deduce \eqref{what we want} for all $L\in \cD$,  with $\iota(L)<\iota(M)$

 Apply Proposition \ref{gggame} to find a subspace $Z$ of $Y_\mathscr{A}$ (i.e., in the image of $P_\mathscr{A}$) that is 1-complemented in $L_1$ via $P:L_1\to Z$ and isometrically isomorphic to $L_1$ via $A:L_1\to Z$. Let also $W$ be the closed linear span of $(\widetilde h_L)_{L\in\mathcal{D}^+}$ in $X$, let $Q:X\to W$ be the canonical 1-projection, and $B:X\to W$ be the canonical onto isometry.

By Proposition \ref{L_1(X) complemented in L_1(X)} the operator $R = ((A^{-1}P)\otimes(B^{-1}Q))S(A\otimes B)$ is a 1-projectional factor of $S$, and thus also of $T$. It remains to see that $R$ is $\e$-close to an $X$-diagonal operator. Fix $L\neq M$. To compute the norm of $R^{(L,M)}$ we also fix $g\in B_{L_\infty}$ and $f\in B_{L_1}$.
\begin{align*}
\big|\big\langle g, R^{(L,M)} f\big\rangle\big| &= \big|\big\langle g\otimes(\nu_Lh_L), R (f\otimes\mu_Mh_M)\big\rangle\big|\\
&= \big|\big\langle \underbrace{P^*A^{-1*}g}_{:=v\in B_{L_\infty}}\otimes\underbrace{(Q^*B^{-1*}\nu_Lh_L)}_{=\nu_L\widetilde h_L}, S (\!\!\!\!\!\underbrace{Af}_{=: u\in B_{Y_\mathscr{A}}}\!\!\!\!\otimes \underbrace{B\mu_Mh_M}_{=\mu_M\widetilde h_M})\big\rangle\big|\\
& = \big|\big\langle v, S^{(\nu_L\widetilde h_L,\mu_M\widetilde h_M)} (P_\mathscr{A}u)\big\rangle\big| \leq \e 2^{-(\iota(L)+\iota(M))}.
\end{align*}
By Remark \ref{in practice close to X-diagonal}, $R$ is $\e$-close to an $X$-diagonal operator.
\end{proof}

\section{Stabilizing entries of $X$-diagonal operators}
Once we have an $X$-diagonal operator at hand we can pass to another $X$-diagonal operator whose entries are stable in an extremely strong sense.
\begin{thm}
\label{stable entries}
Let $X$ be in  $\cH^{**}$  and let $T:L_1(X)\to L_1(X)$ be a bounded $X$-diagonal operator. Then, for any collection of positive real numbers $(\e_L)_{L\in\mathcal{D}^+}$, $T$ is a $1$-projectional factor of an operator $S:L_1(X)\to L_1(X)$ with the following properties:
\begin{enumerate}[leftmargin=20pt,label=(\alph*)]

\item $S$ is $X$-diagonal with entries $(S^L)_{L\in\mathcal{D}^+}$ and

\item for every $L,M\in\mathcal{D}^+$ with $L\subset M$ we have $\|S^L - S^M\| \leq \e_M$.

\end{enumerate}

\end{thm}

Th above theorem is proved in two steps. The first one is to pass, from an arbitrary $X$-diagonal operator, to another one whose entries are uniformly eventually close to Haar multipliers. This is perhaps the most challenging part of the entire process. For presentation purposes we momentarily skip this. Instead, we describe the step that follows it, which is the strong stabilization of the entries, given the uniform eventual proximity to Haar multipliers. This is based on  the Fundamental Lemma (Theorem \ref{fundamental lemma}) and a simple concentration inequality. This proof also serves as an icebreaker for the proof of the first step which is presented afterwards in this section.

\begin{prop}
\label{compact entries to near-tensor}
Let $X$ be in $\cH^{**}$ and let $T:L_1(X)\to L_1(X)$ be a bounded $X$-diagonal operator. Assume that the set of entries $\{T^L:L\in\mathcal{D}^+\}$ of $T$ is uniformly eventually close to Haar multipliers. Then, for any collection of positive real numbers $(\e_L)_{L\in\mathcal{D}^+}$, $T$ is a $1$-projectional factor of an $X$-diagonal operator $S:L_1(X)\to L_1(X)$ so that for every $L,M\in\mathcal{D}^+$ with $L\subset M$ we have $\|S^L - S^M\| \leq \e_M$.
\end{prop}

We start with the probabilistic component required in the proof.

\begin{lem}
\label{concentration splitting}
Let $N\in\N$, $M\geq 0$, $\Omega$ be a uniform probability space with $2N$ elements, and let $\Omega = \sqcup_{n=1}^N\{\omega_n^{-1},\omega_n^1\}$ be a partition of $\Omega$ into doubletons. For a function $G:\Omega\to [-M,M]$ define $\Phi: \{-1,1\}^N\to [-M,M]$ given by
\begin{equation}
\label{splitting sums}
\Phi(\e) = \frac{1}{N}\sum_{n=1}^NG(\omega_n^{\e_n}).
\end{equation}
Then, $\mathbb{E}(\Phi) = \mathbb{E}(G)$ and $\mathrm{Var}(\Phi) \leq M^2/N$, where on $\{-1,1\}^N$ we also consider the uniform probability measure. In particular, for any $\eta >0$,
\begin{equation}
\mathbb{P}\Big(\Big|  \Phi - \mathbb{E}(G)\Big| \geq \eta \Big) \leq\frac{M^2}{N\eta^2}. 
\end{equation}
\end{lem}

\begin{proof}
For $1\leq n\leq N$ let $\Phi_n:\{-1,1\}^N\to [-M,M]$ given by $\Phi_n(\e) = G(\omega_n^{\e_n})$. This is an independent sequence of random variables and for each $n\in\N$ we have
\[\mathbb{E}(\Phi_n) = \frac{1}{2}\big(G(\omega_n^{-1}) + G(\omega_n^{1})\big)\text{ and }\mathrm{Var}(\Phi_n) = \frac{1}{4}\big(G(\omega_n^{-1}) - G(\omega_n^{1})\big)^2.\]
Then, $\mathbb{E}(\Phi) = (1/N)\sum_{n=1}^N\mathbb{E}(\Phi_n) = \mathbb{E}(G)$. By independence we obtain
\begin{equation*}
\mathrm{Var}(\Phi) = \frac{1}{N^2}\sum_{n=1}^N\mathrm{Var}(\Phi_n) \leq \frac{1}{N^2}N\frac{4M^2}{4}= \frac{M^2}{N\eta^2}. 
\end{equation*}
\end{proof}

\begin{lem}
\label{concentration splitting vector valued}
Let $K$ be a relatively compact subset of a Banach space. Then, for every $\e>0$ and $\eta>0$ there exists $N(K,\e,\eta)\in\N$ so that for every $N\geq N(K,\e,\eta)$ the following holds. For every uniform probability space $\Omega$ with $2N$ elements and partition $\Omega = \sqcup_{n=1}^N\{\omega_n^{-1},\omega_n^1\}$ into doubletons, for any function $G:\Omega\to K$, if we define $\Phi: \{-1,1\}^N\to \mathrm{conv}(K)$ given by
\[\Phi(\e) = \frac{1}{N}\sum_{n=1}^NG(\omega_n^{\e_n})\]
then $\mathbb{E}(\Phi) = \mathbb{E}(G)$ and 
\begin{equation}
\mathbb{P}\Big(\Big\|  \Phi - \mathbb{E}(G)\Big\| \geq \eta \Big) \leq\e. 
\end{equation}
\end{lem}

\begin{proof}
The statement $\mathbb{E}(\Phi) = \mathbb{E}(G)$ is proved exactly as in the scalar valued scenario and it is in fact independent of the choice of $N(K,\e,\eta)$. For the second part fix $\e,\eta>0$, and take a finite $\eta/3$-net $(k_i)_{i=1}^{d(K,\eta)}$ of the set $\mathrm{conv}(K\cup(-K))$. Fix norm-one functionals $(f_i)_{i=1}^{d(k,\eta)}$ so that for each $1\leq i\leq d(K,\eta)$ we have $f_i(k_i) = \|k_i\|$. In particular, for any $k_1,k_2\in \mathrm{co}(K)$ with $\|k_1 - k_2\| \geq \eta$ there exists $1\leq i\leq d(K,\eta)$ so that $|f_i(k_1) - f_i(k_2)| \geq \eta/3$. Also set $M = \sup_{k\in K}\|k\|$.

If we now fix $N$, $X$, $G$, and $\Phi$ as in the statement. For $1\leq i\leq d(K,\eta)$ put $G_i = f_i\circ G$ and $\Phi_i = f_i\circ \Phi$ then
\[\big\{ \omega : \|\Phi(\omega) - \mathbb{E}(G)\|>\eta \big\} \subset \bigcup_{1\leq i\leq d(K,\eta)}\big\{ \omega: |\Phi_i(\omega)- \mathbb{E}(G_i) | \geq \eta/3\big\}\]
and thus by Lemma \ref{concentration splitting} we have
\[\mathbb{P}\big(\|\Phi - \mathbb{E}(G)\|>\eta\big) \leq d(K,\eta)\frac{9M^2}{N\eta^2}.\]
Picking any $N(K,\e,\eta) \geq 9d(K,\eta)M^2/(\e\eta^2)$ completes the proof.
\end{proof}

\begin{rem}
\label{sign independence}
Let $X$ be and Haar  system space, $T:L_1(X)\to L_1(X)$ be an $X$-diagonal operator, and $\Gamma$ be a disjoint collection of $\mathcal{D}^+$. Then, for every $g\in L_\infty$, $f\in L_1$, and $\theta$ in $\{-1,1\}^{\Gamma}$ we have
\begin{align*}
\langle g , T^{(\nu_{\Ga^*}h_{\Ga}^\theta,\mu_{\Ga^*}h_\Ga^\theta)}f\rangle &= \langle g\otimes \nu_{\Ga^*} h_\Ga^\theta,T(f\otimes \mu_{\Ga^*}h_\Ga^\theta) \rangle\\
&=|\Ga^*|^{-1}\sum_{M\in\Ga}\sum_{L\in\Ga}\theta_M\theta_L\langle g\otimes h_M,T(f\otimes h_L)\rangle\\
&= |\Ga^*|^{-1}\sum_{M\in\Ga}\sum_{L\in\Ga}\theta_M\theta_L\langle g\otimes h_M,(T^Lf)\otimes h_L\rangle\\
&= |\Ga^*|^{-1}\sum_{M\in\Ga}\sum_{L\in\Ga}\theta_M\theta_L\langle g,T^Lf\rangle\langle h_M,h_L\rangle\\
& = \sum_{L\in\Ga} \big(|L|/|\Ga^*|\big)\langle g,T^Lf\rangle.
\end{align*}
In particular, the above expression does not depend on the choice of signs $\theta$, i.e., we may write
\begin{equation*}
T^\Ga:=T^{(\nu_{\Ga^*}h_{\Ga}^\theta,\mu_{\Ga^*}h_\Ga^\theta)} = \sum_{L\in\Ga}\big(|L|/|\Ga^*|\big)T^L.
\end{equation*}
\end{rem}

\begin{proof}[Proof of Proposition \ref{compact entries to near-tensor}]
  Since $X$ is in  $\cH^{**}$, the conditions of the Fundamental Lemma 
  (Theorem \ref{fundamental lemma}) are satisfied for $B=\{\mu_L h_L: L\in \cD^+\}$. Fix some  $\eta\in (0,1)$.
We apply the Fundamental Lemma to find  a closed subset $\mathscr{A}$ of $[\mathcal{D}^+]$ with $|\mathscr{A}|>1 - \eta$ and
 so that $\{T^{L}P_{\mathscr{A}}: L\in \cD^+\}$ is relatively compact. 
 By Proposition \ref{gggame} there exists a subspace $Z$ of $P_{\mathscr A}(L_1)$ that is isometrically isomorphic to $L_1$ via $A:L_1\to Z$ and 1-complemented in $L_1$ via $P:L_1\to Z$. The operator $T$ is a 1-projectional factor of $S = ((A^{-1}P)\otimes I)T(A\otimes I)$. and in fact for every $L\in \cD^+$ we have $S^{L} = A^{-1}PT^{L}A = A^{-1}PT^{L}P_{\mathscr{A}}A$. In particular, for every  set $\{S^L:L\in \cD^+\}$ is relatively compact.
 Let $K$ be the closed convex hull of $\{S^L:L\in \cD^+\}$, with respect to the operator norm.

As in the proof of Theorem \ref{arbitrary to X-diagonal}, for every finite disjoint collection $\Ga$ of $\mathcal{D}^+$  and $n\!\in\!\N$ define $\Ga(n) = \{L\in\mathcal{D}_n: L\subset \Ga^*\}$. 
For finitely many $n\in\N$, $\Ga(n)$ may be empty however eventually $\Ga^* = \Ga(n)^*$. 
Note that for $n$ sufficiently large so that $\Ga^*  = \Ga(n)^*$ we have
\begin{equation}\label{average}
S^\Gamma_{n} := S^{\Ga(n)} =  \frac{1}{\#\Gamma(n)}\sum_{L\in\Gamma(n)}S^L\in K.\end{equation}
By the relative compactness of $K$  pass to an infinite subset $\mathscr{U}$ of $\N$ so that for each disjoint collection $\Ga$ the limit $S^\Ga_\infty = \lim_{n\in \mathscr{U}}S^\Ga_{n}$ exists.  We point out for later that for any partition $\Ga = \Ga_1\sqcup \cdots\sqcup\Ga_k$ we have
\begin{equation}
\label{compact entries to near-tensor affine}
S^\Ga_\infty = \big(|\Ga_1^*|/|\Ga^*|\big)S^{\Ga_1}_\infty+ \cdots +\big(|\Ga_k^*|/|\Ga^*|\big)S^{\Ga_k}_\infty.
\end{equation}
Pick $(\de_L)_{L\in\mathcal{D}^+}$ so that for all $M\in\mathcal{D}^+$ we have $\sum_{L\subset M}\de_L \leq \e_M/3$. We will recursively define a faithful Haar system $(\hat h_L)_{L\in\mathcal{D}^+}$ so that each $\hat h_L = \sum_{M\in\Ga_L}\ze_Mh_M$ with $\Ga_L\subset \mathcal{D}_{n_L}$, with $n_L\in\mathscr{U}$. We will require that additional conditions are satisfied.

For each $L$ put $\Ga_L^+ = \{M\in\mathcal{D}_{n_L+1}:M\subset[\hat h_\emptyset\hat h_L = 1]\}$ and $\Ga_I^- = \{M\in\mathcal{D}_{n_L+1}:M\subset[\hat h_\emptyset\hat h_L = -1]\}$. In the case $L = \emptyset$ the set $\Ga_\emptyset^-$ is empty and we don't consider it, which is consistent with the fact that there is only one immediate successor of $\emptyset$ in $\mathcal{D}^+$. For each $L$ we define a disjoint collection $E_L$ of $\mathcal{D}^+$ with $E_L^* = \Ga_L^*$. This auxiliary collection $E_L$ will be chosen in the inductive step before $\Ga_L^*$ and in fact it will be used to choose the latter. If $L=\emptyset$ put $E_L = \{[0,1)\}$, if $L = [0,1)$ put $E_L = \Ga_\emptyset$, if $L = L_0^+$ put $E_L = \Ga_{L_0}^+$, and if $L = L_0^-$ put $E_L = \Ga_{L_0}^-$. Below are the additional requirements for each $L\in\mathcal{D}^+$.
\begin{enumerate}[leftmargin=23pt,label=(\roman*)]

\item\label{compact entries to near-tensor 4}  The set $\Ga_L$ is of the form $E_L(n_L)$.

\item\label{compact entries to near-tensor 5} $\|S^{E_L}_{n_L} - S_\infty^{E_L}\| \leq \de_{L}$.

\item\label{compact entries to near-tensor 6}  $\|S_\infty^{E_L} - S_\infty^{\Ga_L^+}\| \leq \de_L$ and $\|S_\infty^{E_L} - S_\infty^{\Gamma_L^-}\| \leq \de_L$.

\end{enumerate}
If we have achieved this construction we define 
\begin{align*}
Q&:L_1(X)\to L_1(X),  
   \text{ by }Q(f)=\!\! \sum_{J,M\in\mathcal{D}^+}\langle h_J\otimes \nu_M\hat h_M, f \rangle |J|^{-1}h_J\!\otimes\! \mu_M\hat h_M,\\
B&:L_1(X)\to L_1(X), \text{ by } B(h_I\otimes h_L)=h_I\otimes\hat h_L.\end{align*}

 Put $R = B^{-1}QSB$. It follows that $R$ is $X$-diagonal and, by Remark \ref{sign independence}, for each $L\in\mathcal{D}^+$ we have $R^L =  S^{E_L}_{n_L}$. Then, for each $L$ we have
\begin{equation}
\label{compact entries to near-tensor telescopic}
\begin{split}
\|R^L - R^{L^+}\|  &= \|S^{E_L}_{n_L} - S^{E_{L^+}}_{n_{L^+}}\| \leq \|S^{\Ga_L^+}_\infty - S^{E_{L^+}}_{n_{L^+}}\| + \|S^{E_L}_{n_L} - S^{\Ga_L^+}_\infty\|\\
&\leq\|S^{\Ga_L^+}_\infty - S^{E_{L^+}}_{n_{L^+}}\| + 2\de_L\text{ (by \ref{compact entries to near-tensor 5} \& \ref{compact entries to near-tensor 6})}\\
&\leq \| S^{\Ga_L^+}_\infty - S^{E_{L^+}}_\infty\| + \|S^{E_{L^+}}_\infty - S^{E_{L^+}}_{n_{L^+}}\| + 2\de_L\\
&\stackrel{\text{\ref{compact entries to near-tensor 5}}}{\leq}   \| S^{\Ga_L^+}_\infty - S^{E_{L^+}}_\infty\| + \de_{L^+} + 2\de_L = 2\de_L + \de_{L^+},
\end{split}
\end{equation}
because, by definition, $\Ga_L^+ = E_{L^+}$. Similarly, we deduce $\|R^L - R^{L^-}\| \leq 2\de_L + \de_{L^{-}}$. Also, using $S^{\Ga_\emptyset} = S^{\Ga_{[0,1)}}$ we deduce $\|R^\emptyset - R^{[0,1)}\| \leq 2\de_\emptyset$. By iterating this process, we may deduce that for every $L\subset M$ we have that $\|R^L - R^M\| \leq 3\sum_{N\subset M}\delta_N\leq \e_M$.

It remains to explain how we ensure that conditions \ref{compact entries to near-tensor 4}, \ref{compact entries to near-tensor 5}, and \ref{compact entries to near-tensor 6} are upheld. We start by putting $E_\emptyset = \{[0,1)\}$, by picking $n_\emptyset$ sufficiently large so that $\|S^{E_\emptyset}_{n_\emptyset} - S^{E_\emptyset}_\infty\| \leq \delta_\emptyset$, and by taking $\zeta_M = 1$ for $M\in E_\emptyset({n_\emptyset}) = \Ga_\emptyset$. Assume that we have carried out the construction up to a certain point and the time has come to pick $\hat h_L$. Let $L_0$ be the immediate predecessor of $L$. We will assume $L = L_0^+$. Similar arguments work if $L = L_0^-$ or if $L = [0,1)$. Put $E_L = \Ga_{L_0}^+$ and pick $n_L\in \mathscr{U}$ so that
 \begin{equation}\label{condition on n_L} \|S^{E_L}_{n_L} - S^{E_L}_\infty\| \leq \de_L\text{ and  }\#E_L(n_L) \geq N(K,1/2,\de_L),\end{equation}
where $N(K,1/2,\de_L)$ is 
 given by Lemma \ref{concentration splitting vector valued}  to the compact set $K$, defined in beginning of this proof. We now apply that  Lemma to $G:E_L(n_L+1)\to K$ with $G(M) = S^{\{M\}}_\infty$. If we endow $G$ with the uniform probability measure, by \eqref{compact entries to near-tensor affine}, $\mathbb{E}(G) = S^{E_L(n_L+1)}_\infty$. Because $E_L(n_L+1)^* = E_L^*$, we may instead write $\mathbb{E}(G) = S^{E_L}_\infty$. We partition $E_L(n_L+1)$ into doubletons by writing $E_L(n_L+1)= \sqcup_{M\in E_L(n_L)}\{M^+,M^{-}\}$. For $M\in E_L(n_L)= \Ga_L$ define $M^1$ and $M^{-1}$ as follows.
\begin{equation*}
M^1 =
\left\{
	\begin{array}{ll}
		M^+  & \mbox{if } M\subset[\hat h_\emptyset = 1]\\
		M^{-} & \mbox{if } M\subset[\hat h_\emptyset = -1]
	\end{array}
\right.\text{ and }
M^{-1} =
\left\{
	\begin{array}{ll}
		M^-  & \mbox{if } M\subset[\hat h_\emptyset = 1]\\
		M^{+} & \mbox{if } M\subset[\hat h_\emptyset = -1]
	\end{array}
\right..
\end{equation*}
Take $\Phi:\{-1,1\}^{\Ga_L}\to\mathrm{conv}(K)$ given by
\begin{equation*}
\Phi(\zeta) = \frac{1}{\#\Ga_L}\sum_{M\in\Ga_L}G(M^{\zeta(M)}) = \frac{1}{\#\Ga_L}\sum_{M\in\Ga_L}S^{\{M^{\zeta(M)}\}}_\infty.
\end{equation*}
By the choice of $n_L$ so that $\#E_L(n_L) \geq N(K,1/2,\de_:)$, there exists a choice $\zeta\in\{-1,1\}^{\Ga_L}$ so that
\begin{equation*}
\|\Phi(\zeta) - \mathbb{E}(G)\| = \|\Phi(\zeta) - S^{E_L}_\infty\| \leq \de_L.
\end{equation*}
By \eqref{compact entries to near-tensor affine} and the definition of $\Phi$ we deduce that $(1/2)(\Phi(\zeta)+ \Phi(-\zeta)) = S^{E_L}_\infty$ and therefore we also have that
\begin{equation*}
\|G(\zeta) - S^{E_L}_\infty\| \leq \de_L.
\end{equation*}
To finish the proof, it remains to observe that if we take $\hat h_L = \sum_{M\in\Ga_L}\zeta(M)h_M$ we have that $S^{\Ga_L^+} = \Phi(\zeta)$ and $S^{\Ga_L^-} = \Phi(-\zeta)$. Indeed, taking a long and hard look at the definition of $M^1$ and $M^{-1}$ we eventually observe that for each $M\in\Ga_L$ we have $(h_\emptyset\zeta(M)h_M)|_{M^{\zeta(M)}} = 1$ and $(h_\emptyset\zeta(M)h_M)|_{M^{-\zeta(M)}} = -1$. This can be seen, e.g., by examining all four possible combinations of values of $h_\emptyset|_M$ and $\zeta(M)$. Therefore, it is now evident that
\begin{align*}
\Ga_L^+ &= \{M\in E_L(n_L+1): M\subset[h_\emptyset h_L = 1]\} = \cup\{M^{\zeta(M)}:M\in\Ga_L\}
\end{align*}
and therefore $S^{\Ga_L^+}_\infty = (\#\Ga_L)^{-1}\sum_{M\in\Ga_L}S_\infty^{\{M^{\zeta(M)}\}} = \Phi(\zeta)$. Finally, by using
\[\frac{1}{2}\Big(S^{\Ga_L^+}_\infty + S^{\Ga_L^-}_\infty\Big) = S^{\Ga_L}_\infty = S^{E_L}_\infty = \frac{1}{2}\Big(\Phi(\zeta) + \Phi(-\zeta)\Big)\]
we see that $S_\infty^{\Ga_L^-} = \Phi(-\zeta)$.

\end{proof}

Now that we are warmed up by the proof of Proposition \ref{compact entries to near-tensor} we are ready to proceed to the slightly more challenging proof of the following. We point out at this point that Theorem \ref{stable entries} is an immediate consequence of Proposition \ref{compact entries to near-tensor} and  the following
Proposition \ref{entries uniformly close to multipliers}

\begin{prop}
\label{entries uniformly close to multipliers}
Let $X$ be in $\cH^{**}$ and $T:L_1(X)\to L_1(X)$ be an $X$-diagonal operator and let $(\e_{(I,J)})_{(I,J)\in(\mathcal{D}^+)^2}$ be a collection of positive real numbers. Then, $T$ is a $1$-projectional factor of an $X$-diagonal operator $S$ with entries $(S_L)_{L\in\mathcal{D}^+}$ and the property that for every $L\in\mathcal{D}^+$ and $I\neq J\in\mathcal{D}^+$ we have
\begin{equation}
\label{ucmycse condition achieved}
\big|\big\langle h_I,S^L\big(|J|^{-1}h_J\big)\big\rangle\big|\leq\e_{(I,J)}.
\end{equation}
In particular, the entries of $S$ are uniformly eventually close to Haar multipliers.
\end{prop}

\begin{proof}
The ``in particular'' part follows from Lemma \ref{uechm in practice} we therefore focus on achieving \eqref{ucmycse condition achieved}.

For each finite disjoint collection $\Gamma$ of $\De^+$ we define $\Ga(n)$, $T^\Ga$, and $T^\Ga_n$ as in the proof of Proposition \ref{compact entries to near-tensor}. Because $X\in\mathcal{H}^{**}$, by Theorem \ref{relatively compact orbit}, applied to the set $B = \{\mu_L h_L:L\in\mathcal{D}^+\}$, for every $f\in L_1$ the set $\{T^Lf:L\in\mathcal{D}^+\}$ is relatively compact and thus so is its convex hull. In particular, for every finite disjoint collection $\Ga$, $\{T^\Ga_nf:n\in\N\}$ is relatively compact 
by \eqref{average}.  By a Cantor diagonalization, we may find  $\mathscr{U}\in[\N]^\infty$ so that $\mathrm{WOT}$-$\lim_{n\in\mathscr{U}}T^\Ga_n = T^\Ga_\infty$ exists for every finite disjoint collection $\Ga$.

We will define inductively two  faithful Haar systems $(\widetilde h_I)_{I\in\mathcal{D}^+}$, $(\widehat h_L)_{L\in\mathcal{D}^+}$. In each step of the induction we will build a single vector $\widetilde h_I$ but we will build an entire level of vectors $\widehat h_L$. For example, in each of the first four steps of the inductive process we will define respectively the collections of vectors
\begin{gather*}
\{\widetilde h_\emptyset; \widehat h_\emptyset\},\; \{\widetilde h_{[0,1)}; \widehat h_{[0,1)}\},\; \{\widetilde h_{[0,1/2)}; \widehat h_{[0,1/2)}, \widehat h_{[1/2,1)}\},\text{ and}\\
\{\widetilde h_{[1/2,1)}; \widehat h_{[0,1/4)}, \widehat h_{[1/4,1/2)}, \widehat h_{[1/2,3/4)}, \widehat h_{[3/4,1)}\}.
\end{gather*}
This asymmetric choice is necessary because whenever we pick a new vector $\widetilde h_I$ we have to stabilize its interaction with all $\widehat h_L$ that will be defined in the future. For each $I,L\in\mathcal{D}^+$ we will have $\widetilde h_I = \sum_{J\in\De_I}h_J$ and $\widehat h_L = \sum_{M\in\Ga_L}\zeta_M h_M$, 
 for some family $(\zeta_M:M\in \Ga_L)\subset\{\pm1\}$.

Let us set up the stage that will allow us to state the somewhat lengthy inductive hypothesis. For each $I\in\mathcal{D}^+$ let
\[\e_I' = \min_{\big\{\substack{J,J'\in\mathcal{D}^+:\\\iota(J),\iota(J')\leq\iota(I)}\big\}}\e_{(J,J')}\]
and fix $(\de_L)_{L\in\mathcal{D}^+}$ so that for all $M\in\mathcal{D}^+$ we have
\[\sum_{L\subset M}\de_L \leq \e'_M/6.\]
Here, $\mathcal{D}_{-1} = \{\emptyset\}$ and $\mathcal{D}_0 = \{[0,1)\}$. For each $k\in\N$ and for every $L\in\mathcal{D}_{k-2}$ we have  for some $n_L\in \mathscr{U}$, $\Ga_L$ is a finite disjoint collection of $\mathcal{D}_{n_L}$ and $|\Ga_L^*| = |L|$. Additionally, if $\iota(I) = k$ the following hold.
\begin{enumerate}[leftmargin=19pt,label=(\alph*)]

\item\label{entries uniformly close to multipliers a} For some $\alpha_k\in\N$, $\De_I$ is a disjoint collection of $\mathcal{D}^{\alpha_k}\setminus\mathcal{D}^{\alpha_{k-1}}$ and $|\De_I| = |I|$. If $k>1$, then $\alpha_k>a_{k-1}$ and we put $\mathcal{D}^{\alpha_0} = \emptyset$.

\item\label{entries uniformly close to multipliers c} For every $J\in\mathcal{D}^+$ with $\iota(J)<k$ and every $M\in\mathcal{D}^{k-2}$ we have
\[\big|\big\langle \widetilde h_I, T^{\Ga_M}\big(|J|^{-1}\widetilde h_J\big)\big\rangle\big| \leq \e_I'/2\text{  and }\big|\big\langle \widetilde h_J, T^{\Ga_M}\big(|I|^{-1}\widetilde h_I\big)\big\rangle\big| \leq \e_I'/2.\]

\end{enumerate}

We will impose additional conditions. As in the proof of Proposition~\ref{compact entries to near-tensor} we put $\Ga_L^+ = \{M\in\mathcal{D}_{n_L+1}:M\subset[\tilde h_\emptyset\tilde h_L = 1]\}$ and $\Ga_L^- = \{M\in\mathcal{D}_{n_L+1}:M\subset[\tilde h_\emptyset\tilde h_L = -1]\}$, for each $L$. If $L=\emptyset$ put $E_L = \{[0,1)\}$, if $L = [0,1)$ put $E_L = \Ga_\emptyset$, if $L = L_0^+$ put $E_L = \Ga_{L_0}^+$, and if $L = L_0^-$ put $E_L = \Ga_{L_0}^-$. Furthermore, for each $\alpha\in\N$ let $P_\alpha:L_1\to L_1$ denote the canonical projection onto $\langle\{h_I:I\in\mathcal{D}^\alpha\}\rangle$. We require the following for each  $L\in\mathcal{D}_{k-2}$.
\begin{enumerate}[leftmargin=23pt,label=(\roman*)]

\item\label{entries uniformly close to multipliers 4}  The set $\Ga_L$ is of the form $E_L(n_L)$.

\item\label{entries uniformly close to multipliers 5} $\|P_{\alpha_k}(T^{E_L}_{n_L} - T_\infty^{E_L})P_{\alpha_k}\| \leq \de_{L}$.

\item\label{entries uniformly close to multipliers 6}  $\|P_{\alpha_k}(T_\infty^{E_L} - T_\infty^{\Ga_L^+})P_{\alpha_k}\| \leq \de_L$ and $\|P_{\alpha_k}(T_\infty^{E_L} - T_\infty^{\Gamma_L^-})P_{\alpha_k}\| \leq \de_L$.

\end{enumerate}

One might jump to the conclusion that the weaker property that, for each $k\in\N$,  $\lim_{n\in\mathscr{U}} P_kT_n^\Ga P_k$ exists is sufficient to yield the same result. This is in fact false. We would not know that $T^\Ga_\infty:L_1\to L_1$ is well defined as the Haar system is not boundedly complete. In the inductive step, the operators $T^{E_L}_\infty$, $L\in\mathcal{D}_{k-2}$ are used in the choice of $\widetilde h_I$, $\iota(I) = k$. Therefore the fact that for each $\Ga$, $\mathrm{WOT}$-$\lim_{n\in\mathscr{U}}T^{\Ga}_n = T^\Ga_\infty$ is necessary. 

We assume that we have completed the construction to finish the proof. Take the isometry $A$ given by $A(h_I\otimes h_L) = \widetilde h_I\otimes\widehat h_L$ and the norm-one projection $P$ onto the image of $A$ given by
\[P(u)= \sum_{I,L\in\mathcal{D}^+}\big\langle\widetilde h_I\otimes |L|^{-1/q}\widehat h_L, u \big\rangle  |I|^{-1}\widetilde h_I\otimes|L|^{-1/p}\widehat h_L.\]
The operator $T$ is a 1-projectional factor of $S = A^{-1}PTA$ and $S$ is $X$-diagonal with entries $(S^L)_{L\in\mathcal{D}^+}$ so that for each $L,I,J\in\mathcal{D}^+$ we have
\[\langle h_I,S^L(h_J)\rangle = \langle \widetilde h_I,T^{\Ga_L}\widetilde h_J\rangle.\]
We fix $I\neq J\in\mathcal{D}^+$ with $\iota(J)<\iota(I) = k$ and $L\in\mathcal{D}^+$. If $L\in\mathcal{D}^{k-2}$ then by \ref{entries uniformly close to multipliers c} we have
\begin{align}
\label{entries uniformly close to multipliers controlled height}
\begin{split}
\big|\big\langle  h_I, S^L\big(|J|^{-1} h_J\big)\big\rangle\big| &\leq \e_I'/2\leq \e_{(I,J)}\text{  and }\\
\big|\big\langle  h_J, S^L\big(|I|^{-1} h_I\big)\big\rangle\big| &\leq \e_I'/2\leq \e_{(J,I)}.
\end{split}
\end{align}
Assume then that $L\in\mathcal{D}_{k'-2}$ with $k' > k$. Let $L_k,\ldots,L_{k'-1},L_{k'} = L$ be a sequence with $L_j\in\mathcal{D}^+_{j-2}$ and each term is a direct successor of the one before it. Repeat the argument from \eqref{compact entries to near-tensor telescopic} to deduce that for $k\leq j<k'$
\begin{gather*}
\big\|P_{\al_k}(T^{\Ga_{L_j}} - T^{\Ga_{L_{j+1}}})P_{\al_k}\big\| \leq \big\|P_{\al_j}(T^{\Ga_{L_j}} - T^{\Ga_{L_{j+1}}})P_{\al_j}\big\| \leq 2\de_{L_j} + \de_{L_{j+1}}\text{, i.e.,}\\
\big\|P_{\al_k}(T^{\Ga_L} - T^{\Ga_{L_{k}}})P_{\al_k}\big\| \leq 3\sum_{M\subset L_k}\de_M \leq \e_{L_k}'/2\leq \e_{I}'/2 \text{ (because }\iota(I)\leq\iota(L_k)\text{).}
\end{gather*}
Therefore,
\begin{align*}
\big|\big\langle  h_I, S^L\big(|J|^{-1} h_J\big)\big\rangle\big| &= \big|\big\langle  \widetilde h_I, T^{\Ga_L}\big(|J|^{-1}\widetilde h_J\big)\big\rangle\big| = \big|\big\langle  \widetilde h_I, P_{\al_k}T^{\Ga_L}P_{\al_k}\big(|J|^{-1}\widetilde h_J\big)\big\rangle\big|\\
 & \leq \big|\big\langle \widetilde h_I, P_{\al_k}T^{\Ga_{L_k}}P_{\al_k}\big(|J|^{-1}\widetilde h_J\big)\big\rangle\big| + \e'_I/2\\
 & = \big|\big\langle \widetilde h_I, T^{\Ga_{L_k}}\big(|J|^{-1}\widetilde h_J\big)\big\rangle\big| + \e'_I/2\\
 & = \big|\big\langle h_I, S^{L_k}\big(|J|^{-1} h_J\big)\big\rangle\big| + \e'_I/2 \stackrel{\eqref{entries uniformly close to multipliers controlled height}}{\leq} \e'_I/2+\e'_I/2 \leq \e_{(I,J)}.
\end{align*}
Repeating the argument yields $|\langle  h_J, S^L(|I|^{-1} h_I)\rangle| \leq \e_{(J,I)}$.
To complete the proof we still need to cary out the inductive construction. In the first step we may take $\widetilde h_\emptyset = h_\emptyset$ (i.e., $\De_\emptyset = \{\emptyset\}$) and thus we may take, e.g., $\al_1 = 1$. Next, we pick $n_\emptyset\in\mathscr{U}$ sufficiently large so that we have $\|P_1(T^{\{[0,1)\}}_{n_\emptyset} - T^{\{[0,1)\}}_\infty)P_1\|\ \leq \de_\emptyset$. We put $\widehat h_\emptyset = \sum_{M\in E_\emptyset(n_\emptyset)} h_L$ (i.e., $\Ga_\emptyset = E_\emptyset(n_\emptyset)$ with $E_\emptyset = \{[0,1)\}$ and $\zeta_M = 1$ for $M\in\Ga_\emptyset$). The only non-trivial condition to check is \ref{compact entries to near-tensor 6}, which follows from the fact that $E_\emptyset^* = (\Ga_\emptyset^+)^*$ and thus  $T^{E_\emptyset}_\infty = T^{\Ga^+_\emptyset}_\infty$. We do not consider the set $\Ga_\emptyset^-$.

We now present the $k$'th step for $k\geq 2$. Let $I\in\mathcal{D}^+$ with $\iota(I) = k$ and denote by $I_0$ its immediate predecessor. We will assume that $I = I_0^+$. For each $L\in\mathcal{D}_{k-2}$ we denote its immediate predecessor by $L_0$. Recall that for each such $L$ the set $E_L$ has been defined based on whether $L = L_0^+$ or $L = L_0^-$. Consider the following finite sets.
\begin{align*}
\mathscr{T} &= \big\{T^{\Ga_L}:L\in\mathcal{D}^{k-3}\big\}\cup\big\{T^{E_L}_\infty:L\in\mathcal{D}_{k-2}\big\}\subset\mathcal{L}(L_1),\\
G&=\big\{\widetilde h_J:\iota(K)<k\big\}\subset L_\infty\text{ and }F = \big\{|J|^{-1}\widetilde h_J:\iota(J)<k\big\}.
\end{align*}
By Lemma \ref{far enough acts small}, there exists $i_0\in\N$ so that for any finite disjoint collection $\De\subset\mathcal{D}^+$ with $\min\iota(\De) \geq i_0$ and any $\theta\in\{-1,1\}^\De$ we have that for all $T\in\mathscr{T}$, $g\in G$, and $f\in F$
\begin{equation}
\label{entries uniformly close to multipliers small on finite}
\big|\big\langle g, T(h^\theta_\De)\big\rangle\big| \leq |I|\e_I'/3\text{ and }\big|\big\langle h^\theta_\De, T(f)\big\rangle\big| \leq \e_I'/3.
\end{equation}
We pick $\De_I$ with $\min\iota(\De_I)\geq i_0$ and so that \ref{entries uniformly close to multipliers a}
is satisfied. The integer $\alpha_k$ is simply chosen so that $P_{\al_k}\widetilde h_I = \widetilde h_I$. It is immediate that condition \ref{entries uniformly close to multipliers c} is satisfied for all $M\in\mathcal{D}^{k-3}$. Later we will show that (b) also holds for $M\in\mathcal{D}^{k-2}$. 

In the next step, for each $L\in\mathcal{D}_{k-2}$ we need to pick $n_L$ that defined $\Gamma_L$ and $\zeta_L\in\{-1,1\}^{\Ga_L}$. The choice of $n_L$ so that \ref{entries uniformly close to multipliers 4} and \ref{entries uniformly close to multipliers 5} are satisfied is easy. However, we wish to ensure that we can additionally achieve condition \ref{entries uniformly close to multipliers 6} and for this we need Lemma \ref{concentration splitting vector valued}.  Consider the relatively compact set $K = \{P_{\al_k}T_\infty^{\Ga}P_{\al_k}: \Ga$ is a finite disjoint collection of $\mathcal{D}^+\}\subset\mathcal{L}(L_1)$ and take $N(K,2^{-k},\e_I'/6)$ given by Lemma \ref{concentration splitting vector valued}. For each $L\in\mathcal{D}^{k-2}$ pick $n_L\in\mathscr{U}$ so that \ref{entries uniformly close to multipliers 5} is satisfied as well as $\#E_L(n_L) \geq N(K,2^{-k},\e_I'/6)$. The objective is to pick, for each $L\in\mathcal{D}_{k-2}$, signs $\zeta_L\in\{-1,1\}^{\Ga_L}$ so that \ref{entries uniformly close to multipliers 6} is satisfied. Repeating, word for word, the argument from the last few paragraphs of the proof of Proposition \ref{compact entries to near-tensor} we can do exactly that.

The final touch that is required to complete the proof is to observe that \ref{entries uniformly close to multipliers c} is now also satisfied for all $L\in\mathcal{D}_{k-2}$. Indeed, for $J\in\mathcal{D}^+$ with $\iota(J) < k$ we have
\begin{align*}
\big|\big\langle  \widetilde h_I, T^{\Ga_L}\big(|J|^{-1}\widetilde h_J\big)\big\rangle\big| &= \big|\big\langle  P_{\al_k}^*\widetilde h_I, T^{E_L}_{n_L}P_{\al_k}\big(|J|^{-1}\widetilde h_J\big)\big\rangle\big|\\
 &= \big|\big\langle\widetilde  h_I, P_{\al_k}T^{E_L}_{n_L} P_{\al_k}\big(|J|^{-1}\widetilde h_J\big)\big\rangle\big|\\
 & \stackrel{\text{\ref{entries uniformly close to multipliers 5}}}{\leq} \big|\big\langle\widetilde h_I, P_{\al_k}T^{E_L}_\infty P_{\al_k}\big(|J|^{-1}\widetilde h_J\big)\big\rangle\big| + \de_L\\
 & \stackrel{\text{\eqref{entries uniformly close to multipliers small on finite}}}{\leq} \e_I'/3 + \de_L \leq \e_I'/3 + \e_I'/6 = \e_I'/2.
\end{align*}
The same argument yields $|\langle \widetilde h_J,T^{\Ga_L}(|I|^{-1}\widetilde h_I)\rangle| \leq \e_I'/2$.
\end{proof}

\section{Projectional factors of scalar operators}
In this section we put the finishing touches to prove our main result.
\begin{thm}
\label{main theorem}
Let $X$ be in $\cH^*$ and $\cH^{**}$ and let $T:L_1(X)\to L_1(X)$ be a bounded linear operator. Then, for every $\e>0$, $T$ is a 1-projectional factor with error $\e$ of a scalar operator. In particular, $L_1(X)$ is primary.
\end{thm}

We first need to prove a perturbation result that will allow us to pass from Theorem \ref{stable entries} to the conclusion.

\begin{prop}
\label{T tensor identity}
Let $X$ be  a Haar  system space and $T:L_1(X)\to L_1(X)$ be an $X$-diagonal operator with entries $(T^L)_{L\in\mathcal{D}^+}$ and let $\e>0$. Assume that for all $L,M\in\mathcal{D}^+$ with $L\subset M$ we have $\|T^L - T^M\| \leq \e|M|^2$. Then, $\|T - T^\emptyset\otimes I\| \leq 7\e$.
\end{prop}

\begin{rem}
Let $n_0\in\N$ and for each $L\in\mathcal{D}_{n_0}$ let $(\theta_k^{L})_{k=0}^{n_0}$ be the signs given by \eqref{normalized intervals simple}. Then, for scalars $(a_L)_{L\in\mathcal{D}_{n_0}}$ we may write
\[\sum_{L\in\mathcal{D}_{n_0}}a_L|L|^{-1}\chi_L = \big(\!\!\!\sum_{L\in\mathcal{D}_{n_0}}\!\!\!a_L\big)h_\emptyset + \sum_{k=1}^{n_0}\sum_{M\in\mathcal{D}_{k-1}}\Big(\!\!\!\!\sum_{\big\{\substack{L\in\mathcal{D}_{n_0}:\\ L\subset M}\big\}}\!\!\!\!\theta_k^La_L\Big)|M|^{-1}h_M.\]

We now take an RI space $X$ and translate this into the $X$ setting. For $k=0,\ldots,n_0$ put $\mu_k = \mu_L$ and $\nu_k = \nu_L$ for $L\in\mathcal{D}_k$. Multiply both sides by $\nu_{n_0}^{-1}$ so that for $L\in\mathcal{D}_{n_0}$ we have $|L|^{-1}\nu_{n_0}^{-1} = \mu_L$.
\begin{gather}
\label{expand ellpeatwototheen}
\sum_{L\in\mathcal{D}_{n_0}}a_L\mu_L\chi_L =\\
\nu_{n_0}^{-1}\big(\!\!\!\sum_{L\in\mathcal{D}_{n_0}}\!\!\!a_L\big)h_\emptyset + \nu_{n_0}^{-1}\sum_{k=1}^{n_0}\nu_{k-1}\!\!\!\!\sum_{M\in\mathcal{D}_{k-1}}
\Big(\!\!\!\!\sum_{\big\{\substack{L\in\mathcal{D}_{n_0}:\\ L\subset M}\big\}}\!\!\!\!\theta_k^La_L\Big)\mu_Mh_M.\nonumber
\end{gather}
Scalar multiplication may be replaced with tensor multiplication to obtain the same formula (i.e., consider $a_L\otimes \chi_L$ where $a_L$ is, e.g., in $L_1$).

Let us additionally observe that for any $1\leq k\leq n_0$ and $M\in\mathcal{D}_{k-1}$ we have
\begin{align}\label{like Holder}
\sum_{\big\{\substack{L\in\mathcal{D}_{n_0}\\L\subset M}\big\}}|a_L|& = \big\langle\sum_{\big\{\substack{L\in\mathcal{D}_{n_0}\\L\subset M}\big\}}|a_L|\mu_{L}\chi_L,\sum_{\big\{\substack{L\in\mathcal{D}_{n_0}\\L\subset M}\big\}}\nu_{L}\chi_L\big\rangle\\
&\leq \big\|\sum_{\big\{\substack{L\in\mathcal{D}_{n_0}\\L\subset M}\big\}}|a_L|\mu_{L}\chi_L\big\|\big\|\sum_{\big\{\substack{L\in\mathcal{D}_{n_0}\\L\subset M}\big\}}\nu_{L}\chi_L\big\|_{X^*}\nonumber\\
& \leq \big\|\sum_{L\in\mathcal{D}_{n_0}}a_L\mu_{L}\chi_L\big\|\nu_{n_0}\big\|\chi_M\big\|_{X^*}\nonumber\\
& =\nu_{n_0} \nu_{k-1}^{-1}  \big\|\sum_{L\in\mathcal{D}_{n_0}}a_L\mu_{L}\chi_L\big\|.\nonumber
\end{align}
\end{rem}

\begin{proof}[Proof of Proposition \ref{T tensor identity}]
For $n=0,1,\ldots$ consider the auxiliary operator $S_n = \sum_{L\in\mathcal{D}_n}T^L\otimes R^L$, where $R^L:X\to X$ denotes the restriction onto $L$, i.e., $R^Lf = \chi_Lf$. We observe that
\begin{align*}
\|S_n - S_{n+1}\| &=\big\|\sum_{L\in\mathcal{D}_n}T^L\otimes\big(R^{L^+} + R^{L^-}\big) - \sum_{L\in\mathcal{D}_n}\big(T^{L^+}\otimes R^{L^+} + T^{L^-}\otimes R^{L^-}\big)\big\|\\
&\leq \sum_{L\in\mathcal{D}_n}\big(\|T^L - T^{L^+}\|\|R_{L^+}\| + \|T^L - T^{L^-}\|\|R^{L^-}\|\big)\\
&\leq 2\e\sum_{L\in\mathcal{D}_n}|L|^2 = \e 2^{-n+1}.
\end{align*}
In particular, for all $n\in\N$ we have 
\begin{equation}\label{estimate}\|T^\emptyset\otimes I - S_n\| = \|S_0 - S_n\| \leq 4\e.\end{equation}
 By \eqref{projective tensor norm}, to estimate $\|T - T^\emptyset\otimes I\|$ it is sufficient to consider vectors of the form $f\otimes g$, with $f\in B_{L_1}$, $g=\sum_{L\in\mathcal{D}_{n_0}}a_L\mu_L\chi_L$, and $\|\sum_{L\in\mathcal{D}_{n_0}}a_L\mu_L\chi_L\| = 1$. By \eqref{expand ellpeatwototheen}
\begin{align}
\label{T tensor identity eq1}
f\otimes g &= \nu_{n_0}^{-1}\big(\!\!\!\sum_{L\in\mathcal{D}_{n_0}}\!\!\!a_L\big)f\otimes h_\emptyset\\
 &\qquad+ \nu_{n_0}^{-1}\sum_{k=1}^{n_0}\nu_{k-1}\!\!\!\!\sum_{M\in\mathcal{D}_{k-1}}\Big(\!\!\!\!\sum_{\big\{\substack{L\in\mathcal{D}_{n_0}:\\ L\subset M}\big\}}\!\!\!\!\theta_k^La_L\Big)f\otimes\mu_Mh_M.\nonumber
\end{align}
 From \eqref{estimate} it follows that 
$$\|(T - T_\emptyset\otimes I)(f\otimes g)\| \leq 4\e + \|(T - S^{n_0})(f\otimes g)\|.$$
  We next evaluate $T$ and $S_{n_0}$ on $f\otimes g$. Since $T$ is $X$-diagonal we have
\begin{align*}
T(f\otimes g)&\stackrel{\eqref{T tensor identity eq1}}{=}\nu_{n_0}^{-1}\big(\!\!\!\sum_{L\in\mathcal{D}_{n_0}}\!\!\!a_L\big)\big(T^\emptyset f\big)\otimes h_\emptyset\\
 &\qquad+ \nu_{n_0}^{-1}\sum_{k=1}^{n_0}\nu_{k-1}\!\!\!\!\sum_{M\in\mathcal{D}_{k-1}}\Big(\!\!\!\!\sum_{\big\{\substack{L\in\mathcal{D}_{n_0}:\\ L\subset M}\big\}}\!\!\!\!\theta_k^La_L\Big)\big(T^Mf\big)\otimes\mu_Mh_M.
\end{align*}
For the other valuation note that for $L\in\mathcal{D}_{n_0}$ we have $S_{n_0}(f\otimes \mu_L\chi_L) = \big(T^Lf\big)\otimes \mu_L\chi_L$. Therefore,
\begin{align*}
S_{n_0}(f\otimes g) 
&\stackrel{\eqref{expand ellpeatwototheen}}{=} \nu_{n_0}^{-1}\big(\!\!\!\sum_{L\in\mathcal{D}_{n_0}}\!\!\!a_L\big(T^Lf\big)\big)\otimes h_\emptyset\\
 &\qquad+ \nu_{n_0}^{-1}\sum_{k=1}^{n_0}\nu_{k-1}\!\!\!\!\sum_{M\in\mathcal{D}_{k-1}}\Big(\!\!\!\!\sum_{\big\{\substack{L\in\mathcal{D}_{n_0}:\\ L\subset M}\big\}}\!\!\!\!\theta_k^La_L\big(T^Lf\big)\Big)\otimes\mu_Mh_M.
\end{align*}
Therefore,
\begin{align*}
\span\big\|(T - S_{n_0})(f\otimes g)\big\|\\ 
&\leq\nu_{n_0}^{-1}\sum_{L\in\mathcal{D}_{n_0}}|a_L|\underbrace{\|T^\emptyset - T^L\|}_{\leq\e} +\nu_{n_0}^{-1}\sum_{k=1}^{n_0}\nu_{k-1}\!\!\!\!\sum_{M\in\mathcal{D}_{k-1}}\!\sum_{\big\{\substack{L\in\mathcal{D}_{n_0}:\\ L\subset M}\big\}}\!\!\!\!|a_L|\underbrace{\|T^M - T^L\|}_{\leq \e|M|^2}\\
&\stackrel{\eqref{like Holder}}{\leq} \nu_{n_0}^{-1}\e\underbrace{\big\|\!\!\!\sum_{L\in\mathcal{D}_{n_0}}a_L\mu_L\chi_L\big\|}_{=1}\nu_{n_0}\underbrace{\nu_0^{-1}}_{=1}  \\
&\qquad+ \e\nu^{-1}_{n_0}\sum_{k=1}^{n_0}\nu_{k-1}\sum_{M\in\mathcal{D}_{k-1}}|M|^2\big\|\!\!\!\sum_{L\in\mathcal{D}_{n_0}}a_L\mu_L\chi_L\big\|\nu_{n_0}\nu_{k-1}^{-1}\\
&= \e + \e\sum_{k=1}^{n_0}\sum_{M\in\mathcal{D}_{k-1}}|M|^2 = \e + \e\sum_{k=1}^{n_0}\frac{2^{k-1}}{2^{2k-2}} \leq 3\e.
\end{align*}
In conclusion, $\|(T - T^\emptyset\otimes I)(f\otimes g)\| \leq 4\e + 3\e$.
\end{proof}

We give the proof of the main result.

\begin{proof}[Proof of Theorem \ref{main theorem}]
Recall that, by virtue of Proposition \ref{approximate factor transitive}, being an approximate 1-projectional factor is a transitive property, during which the compounded errors are under control. We successively apply Theorem \ref{arbitrary to X-diagonal}, Theorem \ref{stable entries} and Proposition \ref{T tensor identity} to find a bounded linear operator $S:L_1\to L_1$ so that $T$ is a 1-projectional factor with error $\e$ of $S\otimes I:L_1(X)\to L_1(X)$. By Theorem \ref{icebreaker}, $S$ is a 1-projectional factor with error $\e$ of a scalar operator $\la I:L_1\to L_1$ and therefore $S\otimes I:L_1(X)\to L_1(X)$ is a 1-projectional factor with error $\e$ of $\la I:L_1(X)\to L_1(X)$. Finally, $T$ is a 1-projectional factor with error $2\e$ of $\la I:L_1(X)\to L_1(X)$. Thus, our claim follows from Proposition \ref{primarity}.
\end{proof}

\section{Final discussion}
Characterizing the complemented subspaces of $L_1$ and those of $C(K)$
remain the most prominent problems in the study of decompositions of classical Banach spaces. This motivates in particular the study of biparameter spaces, especially those with an $L_1$ or $C(K)$ component. The proof, e.g., of primariness for each such type of space presents a different challenge and therefore an opportunity to extract new information on the structure of $L_1$ or $C(K)$ and their operators. Here is a list of classical biparameter spaces, for which primariness remains unresolved.
\begin{enumerate}[leftmargin=20pt,label=(\alph*)]

\item $L_p(L_1)$ for $1<p\leq \infty$.

\item $L_p(L_\infty)\simeq L_p(\ell_\infty)$ for $1\leq p<\infty$.

\item $\ell_p(C(K))$  for a compact metric space $K$ and $1\leq p\leq \infty$.

\item $L_p(C(K))$ for a compact metric space $K$ and $1\leq p\leq \infty$.

\item $C(K,\ell_p)$
  for a compact metric space $K$ and $1\leq p\leq \infty$.

\item $C(K,L_p)$ for a compact metric space $K$ and $1\leq p< \infty$.

\end{enumerate}
Noteworthily, all  the other biparameter Lebesgue spaces $\ell_p(\ell_q)$\cite{CasazzaKottmanLin1978}, $\ell_p(L_q)$\cite{CAP1}, $L_p(L_q)$ \cite{CAPLL},  $L_p(\ell_q)$\cite{CAPLX},  
and $\ell_\infty(L_q)$\cite{Wark2007}   ($1<p,q<\infty$) are known to be primary. The space $L_1(C[0,1])$ resists the approach of this paper but perhaps some of the tools developed here could be of some use. If this were to be resolved, it is conceivable, that techniques from \cite{lechner:motakis:mueller:schlumprecht:2019} may be useful in transcending the separability barrier to show that $L_1(L_\infty)$ is primary. Such methods may also be useful in the investigation of whether for non-separable RI space $X\neq L_\infty$, $L_1(X)$ is primary. In more generality, one may ask for what types of Banach spaces $X$, the spaces $L_1(X)$,  $L_p(X)$, $H_1(X)$ and $H_p(X)$ are primary.

For any two rearrangement invariant Banach function spaces $X$ and $Y$  on $[0,1]$ one can define the biparameter space  $X(Y)$ as the space of all functions $f:[0,1]^2\to \mathbb C$, for which, 
$f(s,\cdot)\in Y$ for all $s\in [0,1]$, and  $g=g_f: [0,1]\to \R$, $s\mapsto \|f(s,\cdot)\|_Y$ is in $X$. The norm of $f$ in  $X(Y)$ would then  be  $\|f\|_{X(Y)} = \| g_f\|_X$. It would be interesting to formulate general  conditions on $X$ and $Y$, which imply that $X(Y)$ is primary, or has the factorization property  (formulated below) with respect to some basis.

The above list may be expanded to the tri-parameter spaces, in which setting there has been little progress.

It is natural to study general conditions under which an operator T on a Banach space is a factor of the identity. A bounded linear operator $T$ on a Banach space $X$ with a Schauder basis $(e_n)_n$ is said to have {\em large diagonal} if $\inf_n|e_n^*(Te_n)| > 0$. If every operator on $X$ with large diagonal is a factor of the identity then we say that $X$ has the {\em factorization property}. The study of the factorization property and that of primariness are closely related. Our proof does not directly show that the spaces under investigation have the factorization property. We may therefore ask: for what Haar  system spaces $X$ and $Y$ does the biparameter Haar system $(h_I\otimes h_L)_{(I,L)\in\mathcal{D}^+\times\mathcal{D}^+}$ have the factorization property in $X(Y)$?

\bibliographystyle{abbrv}
 \bibliographystyle{alpha} \bibliographystyle{plain}
\bibliography{bibliography}

\end{document}